\documentclass[authoryear,12pt]{elsarticle}

\usepackage{xcolor}
\usepackage{amssymb, amsmath}
\usepackage{amsthm}
\newtheorem{theorem}{Theorem}[section]
\newtheorem{corollary}{Corollary}[theorem]
\newtheorem{lemma}[theorem]{Lemma}
\newtheorem{definition}{Definition}

\journal{ArXiv}

\begin{document}

\begin{frontmatter}

\title{Transforming the Erd\H{o}s-Kac theorem}

\author[inst1]{Kimihiro Noguchi}

\affiliation[inst1]{organization={Department of Mathematics, Western Washington University},
            addressline={516 High Street}, 
            city={Bellingham},
            postcode={98225}, 
            state={Washington},
            country={USA}\\
            Email Address: Kimihiro.Noguchi@wwu.edu\\
            ORCiD: 0000-0002-5904-9568}

\begin{abstract}
Transforming the Erd\H{o}s-Kac theorem provides more flexibility in how the theorem can be utilized as an interval estimate for the prime omega function, which counts the number of distinct prime divisors. Here, we consider a direct transformation by the delta method. Then, we demonstrate that the square-root and three-quarters power asymptotically achieve variance stabilization and an optimal width, respectively. Furthermore, by adjusting the denominator of the theorem, we derive the score interval estimate. To make these interval estimates reliable for small positive integers, we examine performances of various interval estimates for the prime omega function using fuzzy coverage probabilities. The results indicate that the score interval estimate performs well even for small positive integers after training the mean and standard deviation using the prime omega function. Moreover, the Poisson interval estimate is relatively reliable with or without training. Additional theoretical results on the Erd\H{o}s-Pomerance theorem are also provided.
\end{abstract}

\begin{keyword}
Box-Cox transformation \sep Fuzzy interval estimate \sep Poisson interval estimate \sep Score interval estimate \sep Variance-stabilizing transformation
\MSC 11N40 \sep 62E20 \sep 60F05
\end{keyword}
\end{frontmatter}

\section{Introduction}
\label{sec:introduction}
One of the landmark results in probabilistic number theory is that of Erd\H{o}s and Kac \citep{erdos1939gaussian, erdos1940gaussian}, whose result shows that certain additive functions follow the Gaussian law (i.e., the normal distribution). As a consequence, they deduce that the prime omega function, $\omega(m)$, $m \in \mathbb{N}$, which counts the number of distinct prime divisors of $m$, is normally distributed as $m \to \infty$, loosely speaking. Formally, let $I(\cdot)$ and $\Phi(\cdot)$ denote the indicator function and the distribution function of the standard normal distribution (i.e., $\mathcal{N}(0, 1)$), respectively. Then,
\begin{align}\label{eqn:ek}
\lim_{n \to \infty}\frac{1}{n}\sum_{m=3}^n I\left(\frac{\omega(m) - \ell_2(m)}{\sqrt{\ell_2(m)}} \leq x\right) = \Phi(x),
\end{align}
$x \in \mathbb{R}$, where $\ell_2(m) = \log\log(m)$. The sum starts at $m = 3$ to ensure that $\sqrt{\ell_2(m)}$ is real-valued. The above result \eqref{eqn:ek} is a major improvement of the previously known result, where Hardy and Ramanujan showed that the normal order of $\omega(m)$ is $\ell_2(m)$ \citep{hardy1917normal}.

Patrick Billingsley advocated a probabilistic view of the Erd\H{o}s-Kac theorem by utilizing the discrete uniform distribution \citep{billingsley1969central,billingsley1974probability}. To understand the probabilistic view, we let $\mathbb{N}_{c,n} = \{c, c+1, c+2, \ldots, n\}$, where $(c,n)$ is a pair of positive integers satisfying $c \leq n$. Then, \eqref{eqn:ek} is equivalent to
\begin{align}\label{eqn:ek2}
\lim_{n \to \infty}\mathbb{P}_n\left(m \colon \left\{ \frac{\omega(m) - \ell_2(m)}{\sqrt{\ell_2(m)}} \leq x\right\} \;\cap\; \left\{J(m) \in \mathbb{N}_{3,n} \right\}\right) = \Phi(x),
\end{align}
$x \in \mathbb{R}$, where $\mathbb{P}_n$ denotes the probability measure that places mass $1/n$ at each of $1, 2, \ldots, n$ among the first $n$ positive integers, and $J$ denotes an identity function (i.e., $J(m)=m$). In addition, as before, the condition $J(m) \in \mathbb{N}_{3,n}$ is necessary to make $\sqrt{\ell_2(m)}$ real-valued, and it does not affect the limit, noting that $\lim_{n \to \infty}\mathbb{P}_n(m \colon J(m) \in \mathbb{N}_{3,n}) = 1$. Although the condition $J(m) \in \mathbb{N}_{3,n}$ is often omitted from the literature on the Erd\H{o}s-Kac theorem, it is included in this article for the sake of clarity. In addition, as the input variable $m$ is often omitted from the expression in probability theory, \eqref{eqn:ek2} may also be expressed as
\begin{align}\label{eqn:ek2eq}
\lim_{n \to \infty}\mathbb{P}_n\left(\left\{\frac{\omega - \ell_2}{\sqrt{\ell_2}} \leq x\right\} \;\cap\; \left\{J \in \mathbb{N}_{3,n}\right\} \right) = \Phi(x),
\end{align}
$x \in \mathbb{R}$. 

Results on the convergence of a sequence of random variables in probability theory and mathematical statistics are usually expressed concisely with some special notations for convenience. For example, $Z_n \stackrel{d}\longrightarrow \Phi$ indicates that the sequence of cumulative distribution functions of the random variables $Z_n$, say, $\{F_{Z_n}\}$, converges pointwise to that of the standard normal distribution as $n \to \infty$. However, since we are dealing with convergence in distribution under $\mathbb{P}_n$, this type of convergence does not fall under a traditional type of convergence. Hence, it is convenient to have a notation to indicate the type of convergence in distribution we are dealing with. To do so, suppose that there is a sequence of functions $\{f^\star_n\}$ whose corresponding domains are given by $\mathbb{M}_n \subseteq \mathbb{N}_{1,n}$, each of which satisfies $\lim_{n \to \infty}\mathbb{P}_n(J \in \mathbb{M}_n) = 1$. Now, let $F^\star_n(x) = \mathbb{P}_n(f^\star_n \leq x \;\cap\; J \in \mathbb{M}_n)$, $x \in \mathbb{R}$, for all $n \geq n_1$, where $n_1$ is some positive integer. Then, we write $f^\star_n \Rightarrow F^\star$ when $\lim_{n \to \infty} F^\star_n(x) = F^\star(x)$ for all the continuity points $x$ of $F^\star$. Furthermore, when $f^\star_n \equiv f^\star$ and $f^\star \Rightarrow F^\star$, we say that $f^\star$ \emph{has} distribution $F^\star$ \citep{billingsley1974probability}. For example, \eqref{eqn:ek2eq} simply becomes 
\begin{align}\label{eqn:ek3}
\frac{\omega - \ell_2}{\sqrt{\ell_2}} \Rightarrow \Phi
\end{align}
with $\mathbb{M}_n = \mathbb{N}_{3,n}$ and $n_1 = 3$ so that $(\omega - \ell_2)/\sqrt{\ell_2}$ has the standard normal distribution. 

The notation ``$\Rightarrow$'' itself does not explicitly specify the sequence of domains $\{\mathbb{M}_n\}$ associated with the convergence. Thus, there is a concern that the domains associated with the convergence may become invalid when more than one function is involved in the convergence given by ``$\Rightarrow$''. However, it turns out that this is not a concern as long as we take the unions of the domains in the sequences. To give an example, suppose that $f^\star_1 \Rightarrow \Phi$ (with, say, $\mathbb{M}_{1,n}$ for $n \geq n_{1,1}$) and $f^\star_2 \Rightarrow 0$ (with, say, $\mathbb{M}_{2,n}$, for $n \geq n_{1,2}$) in the sense that
\begin{align*}
    \lim_{n \to \infty} \mathbb{P}_n\left(\left\{\left\lvert f^\star_2\right\rvert \leq \varepsilon\right\} \;\cap\; \{J \in \mathbb{M}_{2, n}\}\right) = 1
\end{align*}
for any $\varepsilon > 0$. Now, Slutsky's theorem implies that $f^\star_1 + f^\star_2 \Rightarrow \Phi$, which is taken to mean that the domains associated with the convergence are given by $\mathbb{M}_{3,n} = \mathbb{M}_{1,n} \cap \mathbb{M}_{2,n}$ for $n \geq n_{1,3}$, where $n_{1,3} = \max\{n_{1,1}, n_{1,2}\}$. In fact, it is straightforward to show that $\lim_{n \to \infty}\mathbb{P}_n(J \in \mathbb{M}_{3, n}) = 1$ since
\begin{align*}
1 \geq \lim_{n \to \infty}\mathbb{P}_n(J \in \mathbb{M}_{3, n}) \geq \lim_{n \to \infty}\mathbb{P}_n(J \in \mathbb{M}_{1, n}) + \lim_{n \to \infty}\mathbb{P}_n(J \in \mathbb{M}_{2, n}) - 1 = 1
\end{align*}
by applying the additive law of probability.

At the same time, $\{\mathbb{M}_{3, n}\}$, $n \geq n_{1,3}$, can be augmented while preserving the same convergence result. To understand how it works, take any $\{\mathbb{M}_{4, n}\}$, $n \geq n_{1,4}$, $n_{1,3} \geq n_{1,4}$, for which $\mathbb{M}_{3, n} \subseteq \mathbb{M}_{4, n} \subseteq \mathbb{N}_{1,n}$, $n \geq n_{1,3}$, and assume that the $\mathbb{M}_{4, n}$ represents the domain of $f^\star_1 + f^\star_2$ for all $n \geq n_{1,4}$ as well. Then, because $\lim_{n \to \infty}\mathbb{P}_n(J \in \mathbb{M}_{4, n}\backslash \mathbb{M}_{3,n}) = 0$, it is also valid to conclude that $f^\star_1 + f^\star_2 \Rightarrow \Phi$ with $\mathbb{M}_{4, n}$ for $n \geq n_{1,4}$. 

A number-theoretic application of \eqref{eqn:ek3} can be found in construction of an interval estimate of $\omega$ in the vicinity of a sufficiently large $m$, especially on the $\ell_2$ scale. For example, using a two-sided interval estimate of $\omega$ in the vicinity of a sufficiently large $m$, Billingsley claimed that about $60\%$ of the integers in the vicinity of $m = 10^{70}$ have $3$ to $7$ distinct prime divisors \citep{billingsley1969central}. Whether his claim remains valid depends on if $10^{70}$ is considered sufficiently large. From the computational point of view, $10^{70}$ is considered extremely large, making it difficult to directly verify his claim. At the same time, $10^{70}$ may be considered relatively small on the $\ell_2$ scale, noting that $\ell_2(10^{70}) \approx 5$. Thus, that raises the question of whether Billingsley's claim is trustworthy as an interval estimate as it may not be reliable when $\ell_2(m)$ is relatively small. Furthermore, if the interval estimate turns out to be unreliable, then we would like to know whether it is possible to obtain a better interval estimate without imposing an arduous computational burden. These questions naturally lead to an exploration of whether a family of interval estimates for $\omega$ can be constructed, and if so, which interval estimate is deemed most reliable.

To answer these questions, we explore different interval estimates for $\omega$ by following \cite{politis2024studentization}. Specifically, the arguments made in \cite{politis2024studentization} implies that, in addition to the usual standardization as we see in \eqref{eqn:ek3}, there exist two more general strategies to improve the reliability of interval estimates for $\omega$ even for relatively small positive integers. The first strategy is called variance stabilization, in which the variance of a suitably transformed $\omega$ no longer depends on $m$. In other words, the variance becomes a constant. The other strategy, which he calls the ``confidence region'' method, is similar to the score interval estimate approach suggested in \cite{wilson1927probable} and works as follows. Typically, some standardization is applied to a sequence of functions or random variables of interest after centering, such as dividing $\omega - \ell_2$ by $\sqrt{\ell_2}$, to obtain the (asymptotic) distribution of the sequence. However, in the case where the center and standardizer (such as $\ell_2$ and $\sqrt{\ell_2}$, respectively) both depend on a common function, by letting the standardized sequence equal to a constant, solving for the equation for the sequence naturally provides a lower and upper bound of the function of interest. 

As the variance stabilization strategy requires a transformation on $\omega$, it is natural to consider a family of Erd\H{o}s-Kac-type results when a transformation on $\omega$ is applied. However, instead of directly applying the transformation to $\omega$, we consider applying the transformation to $\omega/\ell_2$ instead. This idea comes from the famous Hardy-Ramanujan theorem  \citep{hardy1917normal}, which shows that the normal order of $\omega(m)$ is $\ell_2(m)$, i.e., $\omega/\ell_2 \Rightarrow 1$, in the sense that
\begin{align*}
    \lim_{n \to \infty} \mathbb{P}_n\left(\left\{\left\lvert\frac{\omega}{\ell_2} - 1\right\rvert \leq \varepsilon\right\} \;\cap\; \{J \in \mathbb{N}_{3, n}\}\right) = 1
\end{align*}
for any $\varepsilon > 0$. That is, when a transformation $g$ whose domain contains some interval around $1$ is applied, heuristically, 
\begin{align}\label{eqn:heuristics}
    g(\omega/\ell_2) - g(1) \approx g^\prime(1)(\omega/\ell_2 - 1)
\end{align}
in the vicinity of $1$ assuming that the derivative $g^\prime(1)$ exists. Therefore, we may naturally expect the distribution of $g(\omega/\ell_2)$ to follow the Gaussian law as well, assuming that $g^\prime(1) \neq 0$.

To make the heuristic argument above more rigorous, we utilize frequently used tools on stochastic convergence in probability theory and mathematical statistics. One of the main tools, called Slutsky's theorem, is useful when the convergence of multiple sequences of functions is considered \citep{serfling2009approximation}. Thankfully, Slutsky's theorem still holds even under $\mathbb{P}_n$ as $n \to \infty$ (see Remark 1 of \cite{billingsley1969central}). The other main tool, known as the delta method \citep{bera2023history}, is an extension of the heuristic argument made in \eqref{eqn:heuristics}. Although no results analogous to the delta method had been discussed under $\mathbb{P}_n$ as $n \to \infty$, we show that it still works, and hence the derivation of the distribution on the transformed $\omega/\ell_2$ is possible. That implies that a family of the Erd\H{o}s-Kac-type results can be derived through the probabilistic paradigm shift presented in \eqref{eqn:ek2eq}. In particular, the main result on the distribution of the the transformed $\omega/\ell_2$ is presented here for the sake of convenience.
\begin{theorem}\label{thm:mainomega}
Let $g$ be a (real-valued) function whose domain contains $[1-\delta, 1 + \delta]$ for some $\delta > 0$ with the assumptions that $g$ is differentiable at $1$ and that $g^\prime(1) \neq 0$. Then,
\begin{align}\label{eqn:mainomega}
\frac{g(\omega/\ell_2) - g(1)}{g^\prime(1)/\sqrt{\ell_2}} \Rightarrow \Phi.
\end{align}
\end{theorem}

To improve the reliability of interval estimation of $\omega$ in the vicinity of some positive integer $m$, we closely examine Erd\H{o}s-Kac-type results based on the Box-Cox transformation of $\omega$. The Box-Cox transformation, which includes the power transformation and the logarithmic transformation as special cases \citep{box1964analysis}, is key to obtaining a concise expression of \eqref{eqn:ek}. In particular, the square-root transformation can be identified as the variance-stabilizing power transformation, allowing us to have a distribution whose variance is independent of $m$. In other words, that allows us to develop simple criteria for determining the rarity of a given positive integer in terms of the number of distinct prime divisors. Also, we show that the width of the interval estimate of $\omega$ in the the vicinity of $m$ can be minimized asymptotically by applying the three-quarters power, implying that an optimization of the Erd\H{o}s-Kac theorem is possible through the Box-Cox transformation. 

It is common practice to assess the reliability of an interval estimate (e.g., confidence interval) through the coverage probability. Here, the coverage probability refers to the actual probability that the interval estimate contains the quantity of interest. If the coverage probability is in the vicinity of the nominal (theoretical) coverage probability, then the interval estimate is called \emph{reliable}. In our case, the coverage probability of the interval estimate at $m$ is given by the probability that the values of $\omega$ are included in the corresponding interval estimates in the vicinity of $m$. However, because of the discrete nature of $\omega(m)$, to assess the reliability of the interval estimates of $\omega$ derived from the Box-Cox transformation and from the confidence region approach, their fuzzy coverage probabilities are examined instead for relatively small $m$ (i.e., $m \in [10^5, 10^{14}]$). Here, the word \emph{fuzzy} is used to indicate that the coverage probability is modified using the idea of fuzzy set theory \citep{geyer2005fuzzy}. The modification is necessary as the coverage probability jumps dramatically every time an integer is included in or excluded from the interval estimate.

Unfortunately, as we shall see in Section \ref{sec:mod}, the fuzzy coverage probabilities of the interval estimates based on the variance stabilization and confidence region approaches using the Erd\H{o}s-Kac-type results tend to be much higher than the nominal coverage probability. Instead, the Poisson interval estimate derived from an extension of the prime number theorem discussed in \cite{landau1900quelques} turns out to be more reliable although its fuzzy coverage probabilities are still somewhat higher than the nominal coverage probability. Therefore, to further improve the reliability of these interval estimates, we train their mean and standard deviation using $\omega(m)$ for smaller $m$ (i.e., $m \in [10^4, 10^6]$). Note that the phrase \emph{training} in statistical modeling refers to estimation of the quantity of interest using available information. Then, we demonstrate numerically that the trained version of the score and Poisson interval estimate becomes more reliable when $m \in [10^5, 10^{14}]$, whereas the interval estimates from the Box-Cox transformation remain somewhat less reliable. 

As $\omega$ is a special case of the additive function, we also derive general results on the Gaussian law for a sequence of transformed additive functions. Then, by exponentiating the sequence of additive functions, we state analogous results for the sequence of positive-valued multiplicative functions as well. As a remark, the type of transformation we examine in this article is clearly different from other types of transformations (such as \cite{loyd2023dynamical}), which consider a one-to-one transformation of the standardized version of $\omega$ rather than $\omega$ itself. Lastly, in addition to the results derived from the transformed version of the Erd\H{o}s-Kac theorem, we present similar results on the transformed version of the Erd\H{o}s-Pomerance theorem \citep{erdos1985normal}, which deals with the distinct number of prime divisors of the Euler's totient function, $\varphi(m) = m\prod_{p \vert m} (1 - p^{-1})$, $m \in \mathbb{N}$, where $p$ denotes (the set of) prime numbers.

This article is organized as follows. In Section \ref{sec:main}, we present general results on the distribution of a sequence of transformed additive functions and discuss a similar result for a sequence of transformed positive-valued multiplicative functions. In addition, we show how Theorem \ref{thm:mainomega} follows from the general result. Then, by referring to the general result, in Section \ref{sec:box}, we apply the Box-Cox transformation to the general result and derive a modification of the Erd\H{o}s-Kac theorem. We also show that the square-root transformation achieves a variance-stabilizing transformation of \eqref{eqn:ek3} as a special case. Furthermore, in Section \ref{sec:interval}, we derive that the three-quarters power is asymptotically optimal when the goal is to minimize the width of the interval estimate of $\omega$ in the the vicinity of $m$, and examine how quickly the optimal power transformation converges to three-quarters. Then, in Section \ref{sec:prob}, we investigate the applicability of the Box-Cox-transformed Erd\H{o}s-Kac theorem as an interval estimate of $\omega$ in the vicinity of $m$ by paying attention to the inclusion and exclusion points. To address the problem of coverage probability jumps after the inclusion and exclusion points, in Section \ref{sec:mod}, we introduce the idea of fuzzy coverage probability as a measure of reliability for interval estimates of $\omega$ in the vicinity of $m$. Then, we demonstrate that the suggested interval estimates tend to have fuzzy coverage probabilities higher than the nominal coverage probability for $m \in [10^5, 10^{14}]$ except for the Poisson interval estimate. Next, in Section \ref{sec:train}, we consider a modification of the suggested interval estimates for $\omega$ by training the mean and standard deviation of $\omega$ for $m \in [10^4, 10^6]$. Using them, we show that both the score and Poisson interval estimates become more reliable after training when the fuzzy coverage probabilities for $m \in [10^5, 10^{14}]$ are compared to the nominal coverage probability. Moreover, in Section \ref{sec:ep}, we derive additional theoretical results on the transformed version of the Erd\H{o}s-Pomerance theorem. Finally, Section \ref{sec:conc} offers some concluding remarks, discussion on the reliability of Billingsley's interval estimate in the vicinity of $10^{70}$, and possible future work.

\section{Some general results}
\label{sec:main}
Let $\{f_n\}$ be a sequence of additive functions. That is, $f_n$ is an arithmetic function satisfying $f_n(m_1m_2) = f_n(m_1) + f_n(m_2)$ whenever positive integers $m_1$ and $m_2$ are relatively prime. Furthermore, for (the set of) prime numbers $p$, let $A_n = \sum_{p \leq n} f_n(p)/p$ and $B^2_n = \sum_{p \leq n} f_n^2(p)/p$. Here, it is assumed that there is a positive integer $n_0$ such that $A_n \neq 0$ and $B_n > 0$ for all $n \geq n_0$. Next, let $g$ be a (real-valued) function whose domain contains $[1-\delta, 1 + \delta]$ for some $\delta > 0$. Furthermore, we assume $g$ to be differentiable at $1$ and that $g^\prime(1) \neq 0$. Now, we start with the following preliminary result, which utilizes Theorem 2.1 of \cite{billingsley1974probability}. To do so, we use the key function:
\begin{align*}
r(x) = \begin{cases}
    \frac{g(x) - g(1)}{x-1} - g^\prime(1), & x \neq 1 \\
    0, & x = 1
  \end{cases},
\end{align*}
so that $r$ also has a domain which contains $[1 - \delta, 1 + \delta]$ and is continuous at $x = 1$ (cf., Theorem 3.1 of \cite{serfling2009approximation}).
\begin{lemma}\label{lem:main}
Assume $f_n(m)/B_n \to 0$ for all $m \in \mathbb{N}$ and $\vert A_n \vert/B_n \to \infty$ as $n \to \infty$. Then, $r(f_n/A_n) \Rightarrow 0$.
\end{lemma}
\begin{proof}
First of all, Theorem 2.1 of \cite{billingsley1974probability} states that, if a sequence of additive functions $\{f_n\}$ is such that $\sup_{n \geq n_0} \vert f_n(m)\vert / B_n < \infty$, then, $(f_n - A_n)/(\psi_n B_n) \Rightarrow 0$ if $\psi_n \to \infty$. Furthermore, $f_n(m)/B_n \to 0$, $m \in \mathbb{N}$, implies $\sup_{n \geq n_0} \vert f_n(m)\vert/B_n < \infty$. Also, the theorem holds even if the assumption $\psi_n \to \infty$ is replaced with $\vert \psi_n \vert \to \infty$. Therefore, by letting $\psi_n = A_n/B_n$, 
\begin{align*}
\frac{f_n}{A_n} - 1 = \frac{f_n - A_n}{\psi_n B_n} \Rightarrow 0.
\end{align*}
Now, if we let
\begin{align*}
    \mathbb{M}_{n,\delta}  = \{m \in \mathbb{N}_{n_0, n} \colon f_n(m)/A_n \in [1-\delta, 1+\delta] \},
\end{align*}
$n \geq n_0$, then the above result implies that $\lim_{n \to \infty}\mathbb{P}_n(J \in \mathbb{M}_{n,\delta}) = 1$. Thus, since $r$ is continuous at $x = 1$, for any $\varepsilon > 0$, there exists $\delta^\star \in (0, \delta)$ so that $\vert x - 1 \vert < \delta^\star$ implies $\vert r(x) - r(1) \vert = \vert r(x) \vert < \varepsilon$. Hence,
\begin{align*}
&\lim_{n \to \infty}\mathbb{P}_n(\{\vert r(f_n/A_n) \vert \geq \varepsilon\} \cap \{J \in \mathbb{M}_{n,\delta}\})\\ 
\leq& \lim_{n \to \infty}\mathbb{P}_n(\{\vert f_n/A_n - 1 \vert \geq \delta^\star\} \cap \{J \in \mathbb{M}_{n,\delta}\}) = 0,
\end{align*}
implying that $r(f_n/A_n) \Rightarrow 0$ holds.
\end{proof}
Next, by following Theorem 3.1 of \cite{billingsley1974probability}, the following general result holds.
\begin{theorem}\label{thm:main}
Assume $f_n(m)/B_n \to 0$ for all $m \in \mathbb{N}$, and $\max_{p \leq n} \vert f_n(p) \vert/B_n \to 0$ and $\vert A_n \vert/B_n \to \infty$ as $n \to \infty$. Then, 
\begin{align}\label{eqn:main}
\frac{g(f_n/A_n) - g(1)}{g^\prime(1)(B_n/A_n)} \Rightarrow \Phi.
\end{align}
\end{theorem}
\begin{proof}
First of all, Theorem 3.1 of \cite{billingsley1974probability} states that, assuming $f_n(m)/B_n \to 0$ for all $m \in \mathbb{N}$ and $\max_{p \leq n} \vert f_n(p) \vert/B_n \to 0$ as $n \to \infty$, $(f_n - A_n)/B_n \Rightarrow \Phi$. Moreover, 
\begin{align*}
\frac{g(f_n(m)/A_n) - g(1)}{g^\prime(1)(B_n/A_n)} - \frac{f_n(m) - A_n}{B_n} = \frac{1}{g^\prime(1)}r\left(\frac{f_n(m)}{A_n}\right)\left(\frac{f_n(m) - A_n}{B_n}\right).
\end{align*}
Furthermore, Lemma \ref{lem:main} and the fact that $X_n \Rightarrow 0$ and $Y_n \Rightarrow \Phi$ together imply $X_nY_n \Rightarrow 0$ under $\mathbb{P}_n$ by Slutsky's theorem show that 
\begin{align}\label{eqn:bridge}
\frac{g(f_n/A_n) - g(1)}{g^\prime(1)(B_n/A_n)} - \frac{f_n - A_n}{B_n} \Rightarrow 0
\end{align}
if $\vert A_n \vert/B_n \to \infty$ as $n \to \infty$. Finally, the fact that $X_n - Y_n \Rightarrow 0$ and $Y_n \Rightarrow \Phi$ implies $X_n \Rightarrow \Phi$ under $\mathbb{P}_n$ by Slutsky's theorem shows that \eqref{eqn:main} holds.
\end{proof}
In other words, Theorem \ref{thm:main} above, which may also be understood as
\begin{align*}
\lim_{n \to \infty}\frac{1}{n}\sum_{m \in \mathbb{M}_{n,\delta}} I\left(\frac{g(f_n(m)/A_n) - g(1)}{g^\prime(1)(B_n/A_n)} \leq x\right) = \Phi(x),
\end{align*}
$x \in \mathbb{R}$, achieves a generalization of Theorem 3.1 of \cite{billingsley1974probability} by applying the delta method to his result with the additional assumptions that $A_n \neq 0$ and $B_n > 0$ for all $n \geq n_0$ and that $\vert A_n \vert/B_n \to \infty$ as $n \to \infty$.

As a remark, even though a sequence of additive functions is considered in Theorem \ref{thm:main}, $f_n \equiv f$ is assumed in many applications such as \eqref{eqn:ek3}. Moreover, general results assuming $f_n \equiv f$ (such as the Kubilius-Shapiro theorem stated in Theorem 12.2 of \cite{elliott1980}) can be further modified by applying the $g$ function in a similar way as Theorem \ref{thm:main}.

As an application of Theorem \ref{thm:main}, let us consider $\{v_n\}$, a sequence of positive-valued multiplicative functions. That is, $v_n$ is an arithmetic function satisfying $v_n(m_1m_2) = v_n(m_1)v_n(m_2)$ whenever positive integers $m_1$ and $m_2$ are relatively prime. Also, let $C_n = (1/\log b)\sum_{p \leq n} \log v_n(p)/p$ and $D_n^2 = (1/\log b)^2\sum_{p \leq n} [\log v_n(p)]^2/p$, and that $C_n \neq 0$ and $D_n > 0$ for all $n \geq n_0$. Moreover, let $h$ be a positive-valued function whose domain contains $[\min\{b^{1-\delta}, b^{1+\delta}\}, \max\{b^{1-\delta}, b^{1+\delta}\}]$ for some $\delta > 0$ and $b \in (0, \infty)\backslash\{1\}$. Furthermore, let $h$ be differentiable at $b$ and that $h^\prime(b) \neq 0$. Then, the following corollary holds by letting $v_n(m) = b^{f_n(m)}$ and $g(x) = h(b^x)$ so that $g(1) = h(b)$ and $g^\prime(1) = h^\prime(b)b\log b$.
\begin{corollary}\label{cor:h}
Assume $\log v_n(m)/D_n \to 0$ for all $m \in \mathbb{N}$ and\\ $\max_{p \leq n} \vert \log v_n(p) \vert/D_n \to 0$, and $\vert C_n \vert/D_n \to \infty$ as $n \to \infty$. Then,
\begin{align*}
\frac{h(v_n^{1/C_n}) - h(b)}{h^\prime(b)b\log b(D_n/C_n)} \Rightarrow \Phi.
\end{align*}
\end{corollary}

Now, let us explore the possibility of replacing $g$ in Theorem \ref{thm:main} with a sequence of transformations $\{g_n\}$. Lemma \ref{lem:main} hints that the key step for replacing $g$ with $\{g_n\}$ lies in the uniform convergence of $\{r_n\}$ to $r$ in $[1-\delta, 1+\delta]$, where
\begin{align*}
r_n(x) = \begin{cases}
    \frac{g_n(x) - g_n(1)}{x-1} - g^\prime_n(1), & x \neq 1 \\
    0, & x = 1
  \end{cases},
\end{align*}
$x \in [1-\delta, 1+\delta]$. To show this, we further assume differentiability of $g$ and $\{g_n\}$ on $[1-\delta, 1+\delta]$ and use the following lemma.
\begin{lemma}\label{lem:main2}
Suppose that $\{g_n\}$ is a sequence of differentiable functions whose derivatives $\{g^\prime_n\}$ converge uniformly to $g^\prime$ on $[1-\delta, 1+\delta]$ and that $\lim_{n \to \infty} g_n(1) = g(1)$. Then, $\{r_n\}$ and $\{g_n\}$ converge uniformly to $r$ and $g$, respectively, on $[1-\delta, 1+\delta]$.
\end{lemma}
\begin{proof}
As the case where $x = 1$ is trivial, let $x \neq 1$. Then, for any $\varepsilon > 0$, there exists $n_2 \in \mathbb{N}$ so that for all $n \geq n_2$,
\begin{align*}
    &\lvert r_n(x) - r(x) \rvert = \left\lvert \frac{g_n(x) - g_n(1)}{x-1} - g^\prime_n(1) -\frac{g(x) - g(1)}{x-1} + g^\prime(1) \right\rvert\\
    \leq& \lvert g^\prime_n(\xi_n) - g^\prime(\xi_n) \rvert + \lvert g^\prime_n(1) - g^\prime(1) \rvert < \varepsilon,
\end{align*}
where $\xi_n \in (\min\{1, x\}, \max\{1, x\})$ by the mean value theorem. Furthermore, the condition $\lim_{n \to \infty} g_n(1) = g(1)$ guarantees the uniform convergence of $\{g_n\}$ to $g$ on $[1-\delta, 1+\delta]$ by Theorem 7.17 of \cite{rudin1976principles}.
\end{proof}

\begin{lemma}\label{lem:main3}
Under the same set of assumptions as Lemma \ref{lem:main}, for $\{g_n\}$ in Lemma \ref{lem:main2}, $r_n(f_n/A_n) \Rightarrow 0$.
\end{lemma}
\begin{proof}
Note that, for any $\varepsilon^\star > 0$, there exist $\delta^\star_2 \in (0, \delta)$ and $n_2 \in \mathbb{N}$ so that $\vert x - 1 \vert < \delta^\star_2$ implies 
\begin{align*}
\vert r_n(x) \vert \leq \vert r_n(x) - r(x) \vert + \vert r(x) \vert < \varepsilon^\star
\end{align*}
for all $n \geq n_2$ by Lemma \ref{lem:main2}. Therefore,
\begin{align*}
&\lim_{n \to \infty}\mathbb{P}_n(\{\vert r_n(f_n/A_n) \vert \geq \varepsilon^\star\} \cap \{J \in \mathbb{M}_{n,\delta}\})\\ 
\leq& \lim_{n \to \infty}\mathbb{P}_n(\{\vert f_n/A_n - 1 \vert \geq \delta^\star_2\} \cap \{J \in \mathbb{M}_{n,\delta}\}) = 0,
\end{align*}
implying that $r_n(f_n/A_n) \Rightarrow 0$ holds.
\end{proof}

\begin{theorem}\label{thm:main2}
Under the same set of assumptions as Theorem \ref{thm:main}, for $\{g_n\}$ in Lemma \ref{lem:main2},
\begin{align}\label{eqn:main2}
\frac{g_n(f_n/A_n) - g_n(1)}{g^\prime_n(1)(B_n/A_n)} \Rightarrow \Phi.
\end{align}
\end{theorem}
\begin{proof}
First of all, note that, 
\begin{align*}
\frac{g_n(f_n(m)/A_n) - g_n(1)}{g^\prime_n(1)(B_n/A_n)} - \frac{f_n(m) - A_n}{B_n} = \frac{1}{g^\prime_n(1)}r_n\left(\frac{f_n(m)}{A_n}\right)\left(\frac{f_n(m) - A_n}{B_n}\right).
\end{align*}
Furthermore, Lemma \ref{lem:main3} and the fact that $\lim_{n \to \infty} g^\prime_n(1) = g^\prime(1) \neq 0$ together imply that 
\eqref{eqn:main2} holds by following the same argument as Theorem \ref{thm:main}.
\end{proof}
Note that a result similar to Corollary \ref{cor:h} may be obtained by replacing $h$ with an appropriate sequence $\{h_n\}$ as well. Also, Theorem \ref{thm:main2} is not a full generalization of Theorem \ref{thm:main} as it requires $g$ and $g_n$, $n \in \mathbb{N}$, to be differentiable on $[1-\delta, 1+\delta]$. For simplicity, in the following sections, the results are derived assuming that the same transformation $g$ is applied instead of a sequence of transformations $\{g_n\}$. 

In the next few sections, we will mainly focus on the results derived from Theorem \ref{thm:mainomega}. For that reason, it is beneficial to see how Theorem \ref{thm:main} is related to Theorem \ref{thm:mainomega}. To do so, let $f_n \equiv \omega$,  $A_n = \ell_2(n) + O(1)$, and $B_n^2 = \ell_2(n) + O(1)$, $n \geq 3$. Note that the results on $A_n$ and $B_n^2$ follow from Mertens' second theorem \citep{billingsley1974probability}. Then, because all the assumptions for Theorem \ref{thm:main} are satisfied, by \eqref{eqn:bridge}, the distributions of  
\begin{align*}
\frac{g(f_n/A_n) - g(1)}{g^\prime(1)(B_n/A_n)} \mbox{ and } \frac{f_n - A_n}{B_n}
\end{align*}
are both $\mathcal{N}(0,1)$. Furthermore, because
\begin{align*}
\frac{f_n - A_n}{B_n} - \frac{\omega - \ell_2(n)}{\sqrt{\ell_2(n)}} \Rightarrow 0,
\end{align*}
which is straightforward to show, and that it is derived in \cite{billingsley1969central} that
\begin{align*}
\frac{\omega - \ell_2}{\sqrt{\ell_2}} - \frac{\omega - \ell_2(n)}{\sqrt{\ell_2(n)}} \Rightarrow 0,
\end{align*}
the fact that
\begin{align*}
\frac{g(\omega/\ell_2) - g(1)}{g^\prime(1)/\sqrt{\ell_2}} - \frac{\omega - \ell_2}{\sqrt{\ell_2}} = \frac{1}{g^\prime(1)}r\left(\frac{\omega}{\ell_2}\right)\left(\frac{\omega - \ell_2}{\sqrt{\ell_2}}\right) \Rightarrow 0
\end{align*}
indicates that
\begin{align*}
\frac{g(\omega/\ell_2) - g(1)}{g^\prime(1)/\sqrt{\ell_2}}\Rightarrow \Phi
\end{align*}
also follows, which is exactly what is presented in Theorem \ref{thm:mainomega}.

In Section \ref{sec:box}, we apply a generalization of the power transformation called the Box-Cox transformation to \eqref{eqn:mainomega} to obtain a practical general expression for the Erd\H{o}s-Kac theorem on the distribution of power- and log-transformed $\omega$. Using that, we demonstrate how the square-root transformation leads to a simple expression as a result of the variance-stabilizing transformation. 

\section{Box-Cox transformation on the Erd\H{o}s-Kac theorem}\label{sec:box}
The Box-Cox transformation is a generalization of the power transformation defined on $x > 0$ as 
\begin{align*}
y_\lambda(x) = \begin{cases}
    \frac{x^\lambda - 1}{\lambda}, & \lambda \neq 0 \\
    \log(x), & \lambda = 0
  \end{cases}.
\end{align*}
By applying Theorem \ref{thm:mainomega} to the case where $g(x) = y_\lambda(x)$ and $\mathbb{M}_{n,\delta} \subseteq \mathbb{N}_{3,n}$ for any $\delta \in (0, 1)$, $n \geq 3$, we obtain the following result.
\begin{theorem}\label{thm:m1} 
When $\lambda \neq 0$,
\begin{align}\label{eqn:m1}
\frac{\omega^\lambda - \ell_2^\lambda}{\lambda\ell_2^{\lambda - 1/2}} \Rightarrow \Phi,
\end{align}
and when $\lambda = 0$,
\begin{align}\label{eqn:m2}
\frac{\log\omega - \log\ell_2}{1/\sqrt{\ell_2}} \Rightarrow \Phi.
\end{align}
\end{theorem}

The proof of Theorem \ref{thm:m1} is omitted for the sake of brevity. Also, because $y_\lambda(\omega(m)/\ell_2(m))$ is real-valued for any $\lambda \in \mathbb{R}$ when $m \geq 3$, we may assume that ``$\Rightarrow$'' in \eqref{eqn:m1} and \eqref{eqn:m2} are under $\mathbb{N}_{3,n}$ for $n \geq 3$, noting that $\lim_{n \to \infty} \mathbb{P}_n(J \in \mathbb{N}_{3,n} \setminus \mathbb{M}_{n,\delta}) = 0$. Now, to obtain a concise expression for the transformed version of the Erd\H{o}s-Kac theorem, we examine the case where $\lambda = 1/2$, which is a natural choice as it makes the denominator of \eqref{eqn:m1} equal to $1/2$. In other words, $\lambda = 1/2$ achieves a variance stabilization of the Erd\H{o}s-Kac theorem. In particular, the square-root transformation implies that
\begin{align*}
2\left(\sqrt{\omega} - \sqrt{\ell_2}\right) \Rightarrow \Phi,
\end{align*}
which is equivalent to
\begin{align*}
\lim_{n \to \infty}\frac{1}{n}\sum_{m=3}^n I\left(\sqrt{\omega(m)} - \sqrt{\ell_2(m)} \leq 0.5x\right) = \Phi(x),
\end{align*}
$x \in \mathbb{R}$. The fact that $\lambda = 1/2$ is chosen is not surprising as $A_n$ and $B^2_n$ are both $\ell_2(n) + O(1)$ in the context of Theorem \ref{thm:main}. Indeed, the square-root transformation is used as a variance-stabilizing transformation for the Poisson distribution and gamma distribution with the unit scale parameter, where the mean equals the variance \citep{yu2009variance,noguchi2024asymptotic}.

As a result of the variance stabilization, the concise expression after the square-root transformation may be used to come up with a simple rule-of-thumb for the proportion of positive integers outside certain ranges in terms of the absolute difference between $\sqrt{\omega(m)}$ and $\sqrt{\ell_2(m)}$. Specifically, Table \ref{tab:emp} implies that integers having the absolute difference greater than $1.5$ are rare.
\begin{table}[htbp]
\centering
\begin{tabular}{|l|c|l|} 
 \hline
 Criterion & Actual Proportion & Interpretation \\ 
 \hline
 $\vert \sqrt{\omega(m)} - \sqrt{\ell_2(m)} \vert > 0.5$ & $0.3173105$ & About $1/3$ \\  
  \hline
 $\vert \sqrt{\omega(m)} - \sqrt{\ell_2(m)} \vert > 1.0$ & $0.0455003$ & About $1/20$ \\
  \hline
  $\vert \sqrt{\omega(m)} - \sqrt{\ell_2(m)} \vert > 1.5$ & $0.0026998$ & About $1/400$\\  
  \hline
  $\vert \sqrt{\omega(m)} - \sqrt{\ell_2(m)} \vert > 2.0$ & $0.0000633$ & About $1/16000$\\   
 \hline
\end{tabular}
\caption{Proportion of positive integers satisfying each of the four different criteria in terms of the absolute difference between $\sqrt{\omega(m)}$ and $\sqrt{\ell_2(m)}$.}
\label{tab:emp}
\end{table}

As a remark, by the symmetry of the normal distribution, the actual proportion without the absolute value can be easily computed by dividing the actual proportion in Table~\ref{tab:emp} by $2$. That said, in practice, the results in Table~\ref{tab:emp} may not hold for a relatively small $m$, in which case, a further refinement of Theorem \ref{thm:m1} for small $m$ would become necessary (see Section \ref{sec:train} for more details).

To achieve a further refinement, it is useful to augment Theorems \ref{thm:mainomega} and \ref{thm:m1} by introducing local adjustment functions. Now, suppose that arithmetic functions $\tilde{f}$ and $\tilde{g}$ are well defined for some $\mathbb{M}_n \subseteq \mathbb{N}_{1,n}$, $n \geq n_1$, for some $n_1 \in \mathbb{N}$. In addition, let $\mathbb{M}_n$ be such that $\lim_{n \to \infty}\mathbb{P}_n(J \in \mathbb{M}_n) = 1$. Then, we write
\begin{align*}
\tilde{g} = o(\tilde{f})\mbox{ pp}
\end{align*}
if
\begin{align*}
\lim_{n \to \infty}\mathbb{P}_n\left(\left\{\left\lvert\frac{\tilde{g}}{\tilde{f}} \right\rvert \leq \varepsilon\right\} \;\cap\; \left\{J \in \mathbb{M}_n\right\}\right) = 1
\end{align*}
for any $\varepsilon > 0$. Note that $\tilde{g} = o(\tilde{f})\mbox{ pp}$ is equivalent to $\tilde{g}/\tilde{f} \Rightarrow 0$. As a remark, the notation ``pp'' (\textit{presque partout}) is chosen here as Chapter III.3 of \cite{tenenbaum2015introduction} uses the ``pp'' notation for the same purpose.

Then, the following generalization with local adjustment functions $f_{\mu}$ and $f_{\sigma}$ (or, $f_{\mu, \lambda}$ and $f_{\sigma, \lambda}$) shows that the normality still holds even after some refinements on the terms related to $\ell_2(m)$ in both the numerator and denominator.

\begin{theorem}\label{thm:m2general} 
Let $f_{\mu} = \ell_2 + o(\sqrt{\ell_2})\mbox{ pp}$ and $f_{\sigma} = \sqrt{\ell_2} + o(\sqrt{\ell_2})\mbox{ pp}$ be some local adjustment functions. Also, let $g$ be a (real-valued) function whose domain contains $[1-\delta, 1 + \delta]$ for some $\delta > 0$ with the assumptions that $g$ is differentiable at $1$ and that $g^\prime(1) \neq 0$. Then,
\begin{align*}
\frac{g(\omega/f_{\mu}) - g(1)}{g^\prime(1)f_{\sigma}/f_{\mu}} \Rightarrow \Phi.
\end{align*}
\end{theorem}
\begin{proof}
First, note that 
\begin{align*}
\frac{g(\omega/f_{\mu}) - g(1)}{g^\prime(1)f_{\sigma}/f_{\mu}} - \frac{\omega - f_{\mu}}{f_{\sigma}} = \frac{1}{g^\prime(1)}r\left(\frac{\omega}{f_{\mu}}\right)\left(\frac{\omega - f_\mu}{f_\sigma}\right).
\end{align*}
Moreover, $\omega/f_{\mu} \Rightarrow 1$ by the Hardy-Ramanujan theorem implies $r(\omega/f_{\mu}) \Rightarrow 0$. Hence,
\begin{align*}
\frac{g(\omega/f_{\mu}) - g(1)}{g^\prime(1)f_{\mu}/f_{\sigma}} \mbox{ and }
\frac{\omega - f_{\mu}}{f_{\sigma}}
\end{align*}
have the same distribution by Slutsky's theorem. Also,
\begin{align*}
\frac{\omega - f_\mu}{f_\sigma} = \left(\frac{\sqrt{\ell_2}}{\sqrt{\ell_2} + o(\sqrt{\ell_2})}\right)\left(\frac{\omega - \ell_2}{\sqrt{\ell_2}} + \frac{o(\sqrt{\ell_2})}{\sqrt{\ell_2}}\right) \mbox{ pp } \Rightarrow \Phi
\end{align*}
by Slutsky's theorem implies that the result holds. 
\end{proof}

The following corollary is a direct application of Theorem \ref{thm:m2general} with the Box-Cox transformation.
\begin{corollary}\label{cor:m2o} 
Let $f_{\mu, \lambda} = \ell_2 + o(\sqrt{\ell_2})\mbox{ pp}$ and $f_{\sigma, \lambda} = \ell_2^{\lambda - 1/2} + o(\ell_2^{\lambda - 1/2})\mbox{ pp}$ be some local adjustment functions. Then, when $\lambda \neq 0$,
\begin{align*}
\frac{\omega^\lambda - f_{\mu, \lambda}^\lambda}{\lambda f_{\sigma, \lambda}} \Rightarrow \Phi,
\end{align*}
and when $\lambda = 0$,
\begin{align*}
\frac{\log\omega - \log f_{\mu, \lambda}}{f_{\sigma, \lambda}} \Rightarrow \Phi.
\end{align*}
\end{corollary} 

Similarly, the following corollary, which replaces $\ell_2$ with $\omega$ in the denominator, also holds.
\begin{corollary}\label{cor:m2} 
Let $f_{\mu, \lambda} = \ell_2 + o(\sqrt{\ell_2})\mbox{ pp}$ and $f_{\sigma, \lambda} = \omega^{\lambda - 1/2} + o(\ell_2^{\lambda - 1/2})\mbox{ pp}$ be some local adjustment functions. Then, when $\lambda \neq 0$,
\begin{align*}
\frac{\omega^\lambda - f_{\mu, \lambda}^\lambda}{\lambda f_{\sigma, \lambda}} \Rightarrow \Phi,
\end{align*}
and when $\lambda = 0$,
\begin{align*}
\frac{\log\omega - \log f_{\mu, \lambda}}{f_{\sigma, \lambda}} \Rightarrow \Phi.
\end{align*}
\end{corollary}
\begin{proof}
By the Hardy-Ramanujan theorem, $\omega/\ell_2 \Rightarrow 1$ so that $(\omega/\ell_2)^{\lambda - 1/2} \Rightarrow 1$, implying that $\omega^{\lambda - 1/2} = \ell_2^{\lambda - 1/2} + o(\ell_2^{\lambda - 1/2})\mbox{ pp}$ holds. Hence, the desired result follows from Corollary~\ref{cor:m2o}.
\end{proof}

These local adjustment functions provide a basis for training the mean and standard deviation of (transformed) $\omega$ to obtain reliable interval estimates for $\omega$ across a wide range of positive integers $m$. Estimation of $f_{\mu, \lambda}$ and $f_{\sigma, \lambda}$ is considered in detail in Section \ref{sec:train}. Specifically, Corollary \ref{cor:m2o} is related to Billingsley's interval estimate with different power transformations on $\omega$, and Corollary \ref{cor:m2} is related to the score interval estimate of $\omega$.

\section{Optimizing the interval estimate}
\label{sec:interval}

Let $z_{\alpha/2} > 0$, $\alpha \in (0, 1)$, be such that $\Phi(z_{\alpha/2}) = 1-\alpha/2$. Based on \eqref{eqn:ek3}, a two-sided $100(1-\alpha)\%$ level interval estimate,
\begin{align}\label{eqn:int0}
\left[\ell_2(m) - z_{\alpha/2}\sqrt{\ell_2(m)}, \ell_2(m) + z_{\alpha/2}\sqrt{\ell_2(m)}\right]
\end{align}
for $\omega$ in the vicinity of a sufficiently large $m$, is suggested in \cite{billingsley1969central}. The derivation is straightforward as \eqref{eqn:ek3} implies that
\begin{align*}
\lim_{n \to \infty}\mathbb{P}_n\left(\left\{\omega \in \left[\ell_2 - z_{\alpha/2}\sqrt{\ell_2}, \ell_2 + z_{\alpha/2}\sqrt{\ell_2}\right]\right\} \;\cap\; \left\{J \in \mathbb{N}_{3,n}\right\}\right) = 1-\alpha.
\end{align*}

To optimize the interval estimate for $\omega$, a typical approach taken is to choose an appropriate $\lambda$ so that the width of the interval estimate for $\omega$ in the vicinity of $m$ is minimized. To derive an optimal $\lambda$, first of all, recall that Theorem \ref{thm:m1} implies that
\begin{align}\label{eqn:int11}
\lim_{n \to \infty}\mathbb{P}_n\left(\left\{\omega \in \left[L_{\lambda, \alpha}, U_{\lambda, \alpha}\right]\right\} \;\cap\; \left\{J \in \mathbb{N}_{3,n}\right\}\right) = 1-\alpha,
\end{align}
where, for a sufficiently large $m$,
\begin{align*}
L_{\lambda, \alpha}(m) = \begin{cases}
    \ell_2(m)[1 - z_{\alpha/2}\lambda/\sqrt{\ell_2(m)}]^{1/\lambda}, & \lambda \neq 0 \\
    \ell_2(m)\exp(-z_{\alpha/2}/\sqrt{\ell_2(m)}), & \lambda = 0
  \end{cases},
\end{align*}
and
\begin{align*}
U_{\lambda, \alpha}(m) = \begin{cases}
    \ell_2(m)[1 + z_{\alpha/2}\lambda/\sqrt{\ell_2(m)}]^{1/\lambda}, & \lambda \neq 0 \\
    \ell_2(m)\exp(z_{\alpha/2}/\sqrt{\ell_2(m)}), & \lambda = 0
  \end{cases}.
\end{align*}

Therefore, for a sufficiently large $m$, the width of the interval estimate as a function of $\lambda$ is given by
\begin{align*}
\Delta_{\alpha,m}(\lambda) = 
  \begin{cases}
    \ell_2(m)[[1 + z_{\alpha/2}\lambda/\sqrt{\ell_2(m)}]^{1/\lambda} - [1 - z_{\alpha/2}\lambda/\sqrt{\ell_2(m)}]^{1/\lambda}], & \lambda \neq 0 \\
    \ell_2(m)[\exp(z_{\alpha/2}/\sqrt{\ell_2(m)}) - \exp(-z_{\alpha/2}/\sqrt{\ell_2(m)})], & \lambda = 0
  \end{cases},
\end{align*}
where $\Delta_{\alpha,m}$ is continuous at $\lambda = 0$. Then, we obtain the following result.
\begin{theorem}\label{thm:optim} 
The width of the two-sided $100(1-\alpha)\%$ level interval estimate $[L_{\lambda, \alpha}(m), U_{\lambda, \alpha}(m)]$ is minimized asymptotically at $\lambda = 3/4$ as $\ell_2(m) \to \infty$.
\end{theorem}
\begin{proof}
An asymptotic expansion of $\Delta_{\alpha,m}(\lambda)$ as $\ell_2(m) \to \infty$ is given by
\begin{align*}
\Delta_{\alpha, m}(\lambda) = 
2z_{\alpha/2}\sqrt{\ell_2(m)}+\frac{z_{\alpha/2}^3}{3\sqrt{\ell_2(m)}}(\lambda-1)(2\lambda-1) + O\left(\frac{1}{[\ell_2(m)]^{3/2}}\right).
\end{align*}
Therefore, by considering $(\lambda- 1)(2\lambda - 1)$, $\lambda = 3/4$ asymptotically minimizes the width of the two-sided $100(1-\alpha)\%$ level interval estimate. 
\end{proof}
In other words, an asymptotically optimal $100(1-\alpha)\%$ level interval estimate for $\omega$ in the vicinity of $m$ is given by
\begin{align}\label{eqn:int34}
&\left[L_{3/4, \alpha}(m), U_{3/4, \alpha}(m)\right] \nonumber\\ &= \left[\ell_2(m)[1 - 3z_{\alpha/2}/(4\sqrt{\ell_2(m)})]^{4/3}, \ell_2(m)[1 + 3z_{\alpha/2}/(4\sqrt{\ell_2(m)})]^{4/3}\right].
\end{align}

Although Theorem~\ref{thm:optim} is an asymptotic result, for a finite $m$, the optimal value of $\lambda$ may vary depending on the value of $m$ and $z_{\alpha/2}$. In general, as implied by the asymptotic expansion of $\Delta_{\alpha,m}(\lambda)$, the optimal value of $\lambda$ converges to $3/4$ quickly for a small value of $z_{\alpha/2}$. However, the convergence to $3/4$ may require a much larger value of $m$ when $z_{\alpha/2}$ is relatively large. 

To see how quickly the optimal value of $\lambda$ converges to $3/4$ for various $m$ and $z_{\alpha/2}$, we only examine cases where the lower bound $L_{\lambda,\alpha}(m)$ is greater than $1$, which is when the interval estimate provides useful information about $\omega$ since $\omega(m) \geq 1$ for all $m$. To this end, we define the following.
\begin{definition}[Inclusion and Exclusion Points]
Let $k \in \mathbb{N}$. We call $U^{-1}_{\lambda,\alpha}(k)$ the $k$-th inclusion point. Similarly, we call $L^{-1}_{\lambda,\alpha}(k)$ the $k$-th exclusion point.
\end{definition}
The inclusion and exclusion points may be easily understood by using the floor and ceiling function, denoted by $\lfloor \cdot \rfloor$ and $\lceil \cdot \rceil$, respectively. For example, $\lceil U^{-1}_{\lambda,\alpha}(k) \rceil$ is the positive integer for which $k$ is included in the interval estimate $[L_{\lambda,\alpha}(m), U_{\lambda,\alpha}(m)]$ for the first time. Similarly, $\lfloor L^{-1}_{\lambda,\alpha}(k) \rfloor$ is the positive integer for which $k$ is included in the interval estimate $[L_{\lambda,\alpha}(m), U_{\lambda,\alpha}(m)]$ for the last time.
When $\lambda = 1$ and $\lambda = 1/2$, closed-form solutions for both $\ell_2(U^{-1}_{\lambda,\alpha}(k))$ and $\ell_2(L^{-1}_{\lambda,\alpha}(k))$ can be easily derived. For instance, when $\lambda = 1$,
\begin{equation*}
\begin{aligned}
\ell_2(U^{-1}_{1,\alpha}(k)) &= (k + z_{\alpha/2}^2/2) - z_{\alpha/2}\sqrt{k + z_{\alpha/2}^2/4},\\
\ell_2(L^{-1}_{1,\alpha}(k)) &= (k + z_{\alpha/2}^2/2) + z_{\alpha/2}\sqrt{k + z_{\alpha/2}^2/4}.
\end{aligned}
\end{equation*}
On the other hand, when $\lambda = 1/2$,
\begin{equation*}
\begin{aligned}
\ell_2(U^{-1}_{1/2,\alpha}(k)) &= \left(\sqrt{k} -  z_{\alpha/2}/2\right)^2,\\
\ell_2(L^{-1}_{1/2,\alpha}(k)) &= \left(\sqrt{k} +  z_{\alpha/2}/2\right)^2,
\end{aligned}
\end{equation*}
where the above expression for $\ell_2(U^{-1}_{1/2,\alpha}(k))$ holds when $z_{\alpha/2}^2 \leq 4k$, corresponding to the fact that $\sqrt{\ell_2(m)}$ must be a real number. Next, let us consider the first exclusion point, $L^{-1}_{\lambda,\alpha}(1)$, which plays an important role as $\omega(m) \geq 1$, $m \in \mathbb{N}$. Specifically, when $[L_{\lambda, \alpha}(m), U_{\lambda, \alpha}(m)]$ is given in the form described in \eqref{eqn:int11}, $\theta = L^{-1}_{\lambda,\alpha}(1)$ is such that
\begin{align*}
z_{\alpha/2} = \begin{cases}
    [\ell_2(\theta)]^{1/2 - \lambda}\left(\frac{[\ell_2(\theta)]^\lambda - 1}{\lambda}\right), & \lambda \neq 0 \\
    \sqrt{\ell_2(\theta)}\log(\ell_2(\theta)), & \lambda = 0
  \end{cases}.
\end{align*}
To give a sense of how the choice of $z_{\alpha/2}$ affects $L^{-1}_{\lambda,\alpha}(1)$, Table \ref{tab:eff} summarizes the results for $\ell_2(L^{-1}_{3/4,\alpha}(1))$ and the corresponding $L^{-1}_{3/4, \alpha}(1)$.
\begin{table}[htbp]
\centering
\begin{tabular}{|c|c|c|} 
 \hline
 $z_{\alpha/2}$ &$\ell_2(L^{-1}_{3/4,\alpha}(1))$ & $L^{-1}_{3/4,\alpha}(1)$ \\ 
 \hline
 1 & $2.4104$ & $6.9\cdot 10^4$ \\  
  \hline
 2 & $4.7422$ & $6.4\cdot 10^{49}$ \\
  \hline
 3 & $8.0830$ & $4.9\cdot 10^{1406}$\\  
  \hline
\end{tabular}
\caption{The value of $\ell_2(L^{-1}_{3/4,\alpha}(1))$ and its corresponding $L^{-1}_{3/4,\alpha}(1)$.}
\label{tab:eff}
\end{table}

Now, let us examine how quickly the optimal value of $\lambda$ converges to $3/4$ by examining how the optimal value of $\lambda$ changes beyond the point at which the lower bound exceeds $1$. For example, when $z_{\alpha/2} = 2$, the left panel of Figure~\ref{fig:lambda} shows that the optimal value of $\lambda$ is in the vicinity of $3/4$ when $\ell_2(m) \geq 5$ (i.e., $m \geq 2.9\cdot 10^{64}$) approximately, noting that $\ell_2(L^{-1}_{3/4,\alpha}(1)) \approx 4.817$ (i.e., $L^{-1}_{3/4,\alpha}(1) \approx 4.7\cdot 10^{53}$). On the other hand, when $z_{\alpha/2} = 3$, the right panel of Figure~\ref{fig:lambda} shows that the optimal value of $\lambda$ is in the vicinity of $3/4$ when $\ell_2(m) \geq 10$ (i.e., $m \geq 9.4\cdot 10^{9565}$) approximately, noting that $\ell_2(L^{-1}_{3/4,\alpha}(1)) \approx 8.383$ (i.e., $L^{-1}_{3/4,\alpha}(1) \approx 6.0\cdot 10^{1898}$).
\begin{figure}
    \centering
    \includegraphics[width=0.45\textwidth]{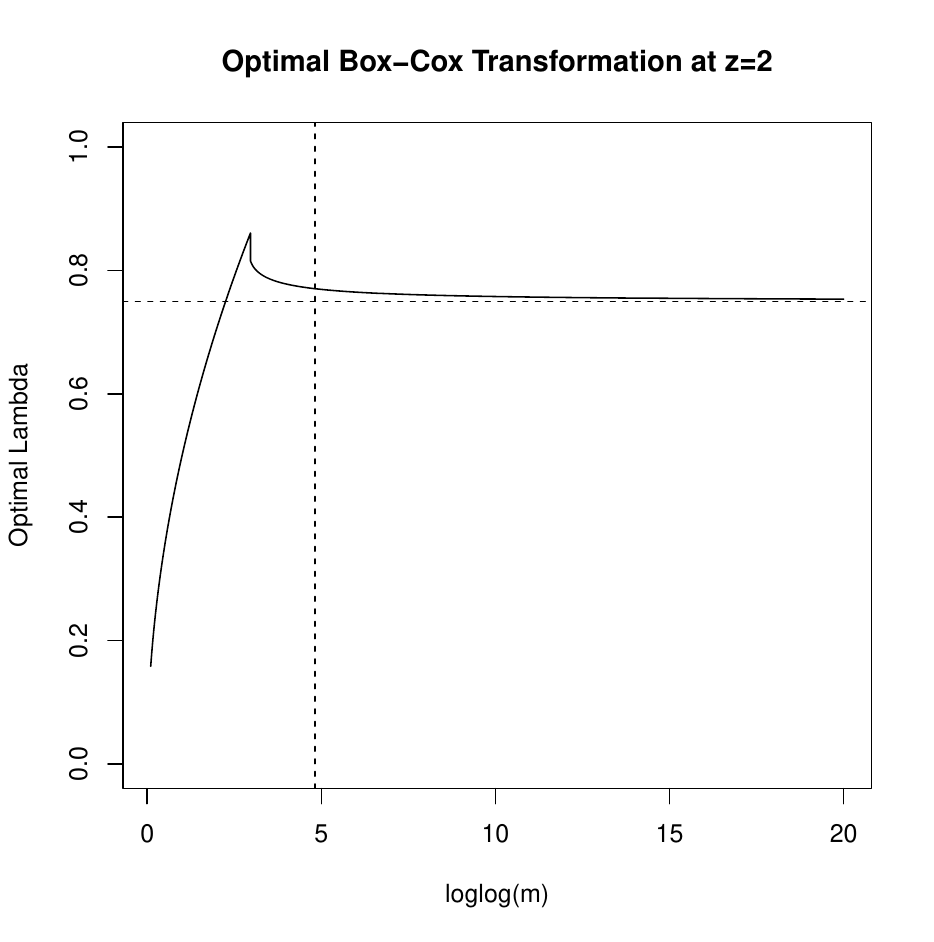}
    \includegraphics[width=0.45\textwidth]{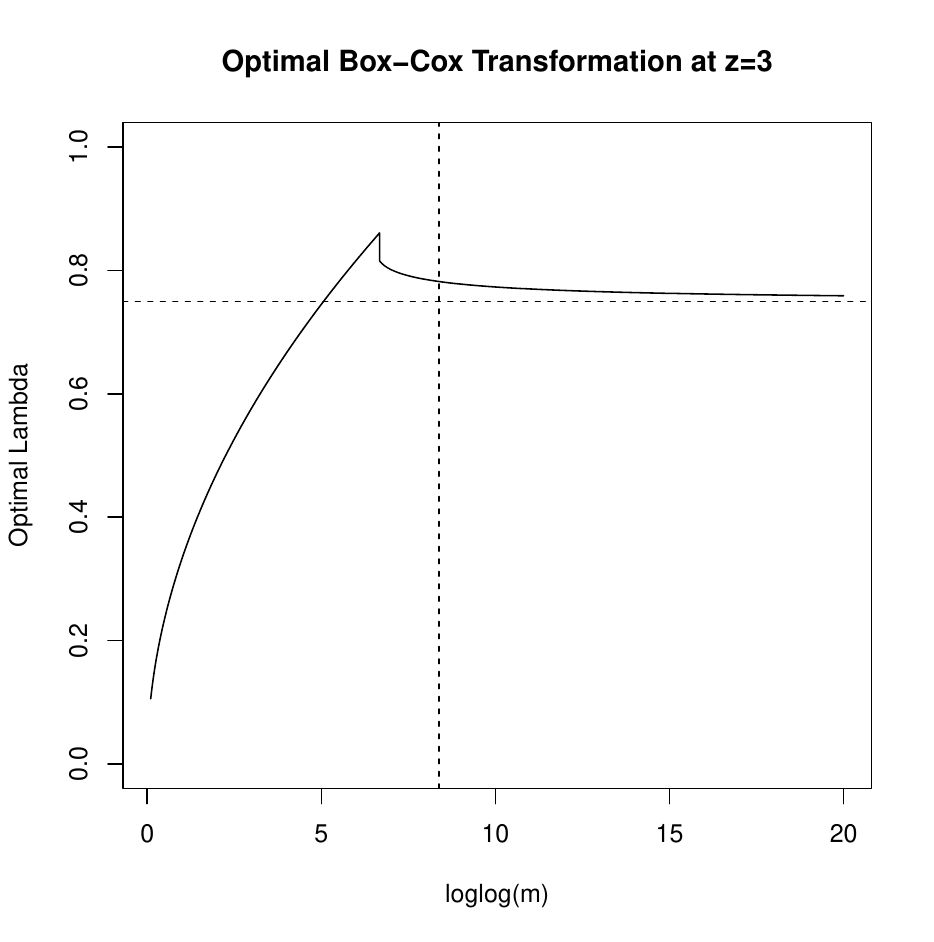}
    \caption{The optimal value of $\lambda$ as a function of $\ell_2(m)$ at $z_{\alpha/2} = 2$ (left) and at $z_{\alpha/2} = 3$ (right). The horizontal dashed line is drawn at $\lambda = 3/4$. The vertical dashed line represents the point at which the lower bound exceeds $1$.}
    \label{fig:lambda}
\end{figure}

\section{Problems with the interval estimates}\label{sec:prob} 
Let $\alpha^\star$ be such that $z_{\alpha^\star/2} = 0.9$, corresponding to $\alpha^\star \approx 0.3681$. An application of the interval estimate \eqref{eqn:int0} shows that $[L_{1,\alpha^\star}(10^{70}), U_{1,\alpha^\star}(10^{70})] \approx [3, 7]$, implying that about $60\%$ of the integers in the vicinity of $10^{70}$ have $3$ to $7$ distinct prime divisors \citep{billingsley1969central}. Curiously, it is also noted in \cite{billingsley1969central} that a computation of this kind is based on the premise that the integers near $10^{70}$ are ``not atypical in their divisibility properties''. Although it is unclear from \cite{billingsley1969central} the situations in which the integers in the vicinity of $m$ are considered \emph{atypical} in their divisibility properties, we describe situations where the interval estimate \eqref{eqn:int34} may become less reliable due to irregularities in the coverage probability. 

Let us recall that the interval estimate of $\omega$ in the vicinity of $m$ is deemed \emph{reliable} when the actual probability (i.e., the coverage probability) that $[L_{\lambda, \alpha}(t), U_{\lambda, \alpha}(t)]$ contains $\omega(t)$, $t \in [m-j, m+j]$ for some $j < m$, is ``close enough'' to the nominal coverage probability $1-\alpha$. That is, by denoting the coverage probability around $m$ by $p_{m, j, \lambda, \alpha}$, i.e.,
\begin{align}\label{eqn:cr}
    p_{m, j, \lambda, \alpha} = \frac{1}{2j+1}\sum_{t=m-j}^{m+j}I_{\lambda, \alpha}(t),
\end{align}
where 
\begin{align*}
    I_{\lambda, \alpha}(t)=I\left(\omega(t) \in \left[L_{\lambda,\alpha}(t), U_{\lambda,\alpha}(t)\right]\right),
\end{align*}
$p_{m, j, \lambda, \alpha} \approx 1 - \alpha$ indicates that the interval estimate is reliable in the vicinity of $m$; otherwise, it is deemed unreliable. Although it is somewhat subjective when we describe what we mean by ``close enough'', Bradley's negligible criterion suggests $1 - 1.1\alpha \leq p_{m, j, \lambda, \alpha} \leq 1 - 0.9\alpha$ in the context of hypothesis testing \citep{bradley1978robustness,ramsey1980exact}. That is, for $\alpha^\star$, $1 - \alpha^\star = 2\Phi(0.9)-1 \approx 0.6319$, corresponding to $0.5951 \leq p_{m, j, \lambda, \alpha^\star} \leq 0.6687$. In other words, if $p_{m, j, \lambda, \alpha^\star}$ is off by at most a few percentage points from the nominal coverage probability $1-\alpha^\star$, then the interval estimate may be deemed reliable according to the criterion. On the other hand, a somewhat less stringent criterion, known as Bradley's liberal criterion, suggests $1 - 1.5\alpha \leq p_{m, j, \lambda, \alpha} \leq 1 - 0.5\alpha$ instead. For $\alpha^\star$, it is equivalent to $0.4478 \leq p_{m, j, \lambda, \alpha^\star} \leq 0.8159$, allowing the coverage probability to vary by close to twenty percentage points on both ends from the nominal coverage probability.

To describe the general behavior of the interval estimate, we demonstrate that its coverage probability is greatly affected by the inclusion and exclusion points, as well as $m$ itself. For convenience, we examine the behavior of $[L_{1, \alpha^\star}(m), U_{1, \alpha^\star}(m)]$. In other words, the values of $\lambda$ and $z_{\alpha/2}$ are made consistent with those used in \cite{billingsley1969central}.

\begin{table}[htbp]
\centering
\begin{tabular}{|c|c|c|c|c|} 
 \hline
 $k$ & $\ell_2(L^{-1}_{1,\alpha^\star}(k))$ & $L^{-1}_{1,\alpha^\star}(k)$ & $\ell_2(U^{-1}_{1,\alpha^\star}(k))$ & $U^{-1}_{1,\alpha^\star}(k)$\\ 
 \hline
 1 & $0.4181$ & $4.6\cdot 10^0$ & $2.3919$ & $5.6\cdot 10^4$\\  
  \hline
 2 & $1.0693$ & $1.8\cdot 10^1$ & $3.7407$ & $2.0\cdot 10^{18}$\\
  \hline
 3 & $1.7944$ & $4.1\cdot 10^2$ & $5.0156$ & $2.9\cdot 10^{65}$\\  
  \hline
 4 & $2.5600$ & $4.1\cdot 10^5$ & $6.2500$ & $9.3\cdot 10^{224}$\\  
  \hline
 5 & $3.3522$ & $2.5\cdot 10^{12}$ & $7.4578$ & $6.1\cdot 10^{752}$\\
  \hline
 6 & $4.1636$ & $8.4\cdot 10^{27}$ & $8.6464$ & $1.1\cdot 10^{2471}$\\  
  \hline 
 7 & $4.9896$ & $6.2\cdot 10^{63}$ & $9.8204$ & $1.4\cdot 10^{7993}$\\  
  \hline
 8 & $5.8274$ & $2.7\cdot 10^{147}$ & $10.9826$ & $3.1\cdot 10^{25554}$\\ 
  \hline 
\end{tabular}
\caption{The first $8$ inclusion and exclusion points of $[L_{1, \alpha^\star}(m), U_{1, \alpha^\star}(m)]$.}
\label{tab:ie}
\end{table}

\begin{table}[htbp]
\centering
\begin{tabular}{|c|c|c|c|c|} 
 \hline
 Phase & CP & Diff.\ & Asy.\ CP & Asy.\ Diff.\ \\ 
 \hline
 1 & 0.8090 & -- & 0.5727 & -- \\  
  \hline
 2  & 0.7153 & -0.0937 & 0.4798 & -0.0929\\
  \hline
 3  & 0.6651 & -0.0502 & 0.4513 & -0.0285\\ 
  \hline
 4  & 0.8755 & 0.2104 & 0.6673 & 0.2160\\  
  \hline 
\end{tabular}
\caption{CP: coverage probabilities. Diff.: Differences between the coverage probabilities. Asy.\ CP: coverage probabilities based on the asymptotic formula. Asy.\ Diff.: Differences between coverage probabilities based on the asymptotic formula. Phase 1: $ m = \lfloor U^{-1}_{1,\alpha^\star}(1)-j\rfloor$, Phase 2: $m = \lceil U^{-1}_{1,\alpha^\star}(1)+j \rceil$, Phase 3: $m = \lfloor L^{-1}_{1,\alpha^\star}(4)-j \rfloor$, Phase 4: $m = \lceil L^{-1}_{1,\alpha^\star}(4)+j \rceil$. The coverage probabilities based on the asymptotic formula are computed at the inclusion and exclusion points.}
\label{tab:iedelta}
\end{table}

Table~\ref{tab:ie} shows the results for the first $8$ inclusion and exclusion points of $[L_{1, \alpha^\star}(m), U_{1, \alpha^\star}(m)]$. Interestingly, Figure~\ref{fig:CR} displays that the first inclusion point, $U^{-1}_{1,\alpha^\star}(1) \approx 5.6\cdot 10^4$, corresponds to the point where we see a noticeable increase in the coverage probability. Similarly, the fourth exclusion point, $L^{-1}_{1,\alpha^\star}(4) \approx 4.1\cdot 10^5$, following the first inclusion point corresponds to a noticeable decrease in the coverage probability. To see the magnitudes of the increase and decrease in the coverage probabilities corresponding to these points, Table~\ref{tab:iedelta} shows the computed coverage probabilities and their differences before and after these inclusion and exclusion points. Furthermore, these results are compared to the coverage probabilities computed from the asymptotic formula below. To be more specific, let $\rho_d(m)$ be the number of integers not exceeding $m$ with exactly $d$ distinct prime divisors. Then, the values of the estimated coverage probabilities can be computed using the asymptotic formula for $\rho_d(m)$ given by 
\begin{align}\label{eqn:asymprop}
\hat{\rho}_d(m) = \frac{m\exp(-\ell_2(m))[\ell_2(m)]^{d-1}}{(d-1)!},
\end{align}
where $\hat{\rho}_d(m) \sim \rho_d(m)$ as $m \to \infty$ \citep{landau1900quelques}. The formula \eqref{eqn:asymprop} is a well-known generalization of the prime number theorem as $\hat{\rho}_1(m) = m/\log(m)$. By utilizing \eqref{eqn:asymprop}, the estimated coverage probability is given by
\begin{align}\label{eqn:asympcr}
p^\star_{m,\lambda,\alpha} = \frac{1}{m}\sum_{d=1}^\infty \hat{\rho}_d(m)I_{d, \lambda, \alpha}(m),
\end{align}
where
\begin{align*}
    I_{d, \lambda, \alpha}(m) = I(d \in [L_{\lambda, \alpha}(m), U_{\lambda, \alpha}(m)]).
\end{align*}
Table \ref{tab:iedelta} shows that \eqref{eqn:asympcr} well approximates the differences in the coverage probabilities before and after the inclusion and exclusion points, but it does not well approximate the coverage probability itself due to a relatively small $m$ being used. As a remark, $\hat{\rho}_d(m)/m$ is equivalent to the probability mass function of the shifted Poisson distribution with parameter $\ell_2(m)$ with the minimum of $1$ instead of the usual $0$ due to $d-1$ in the formula \citep{kowalski2021introduction}.

\begin{figure}
    \centering
    \includegraphics[width=0.7\textwidth]{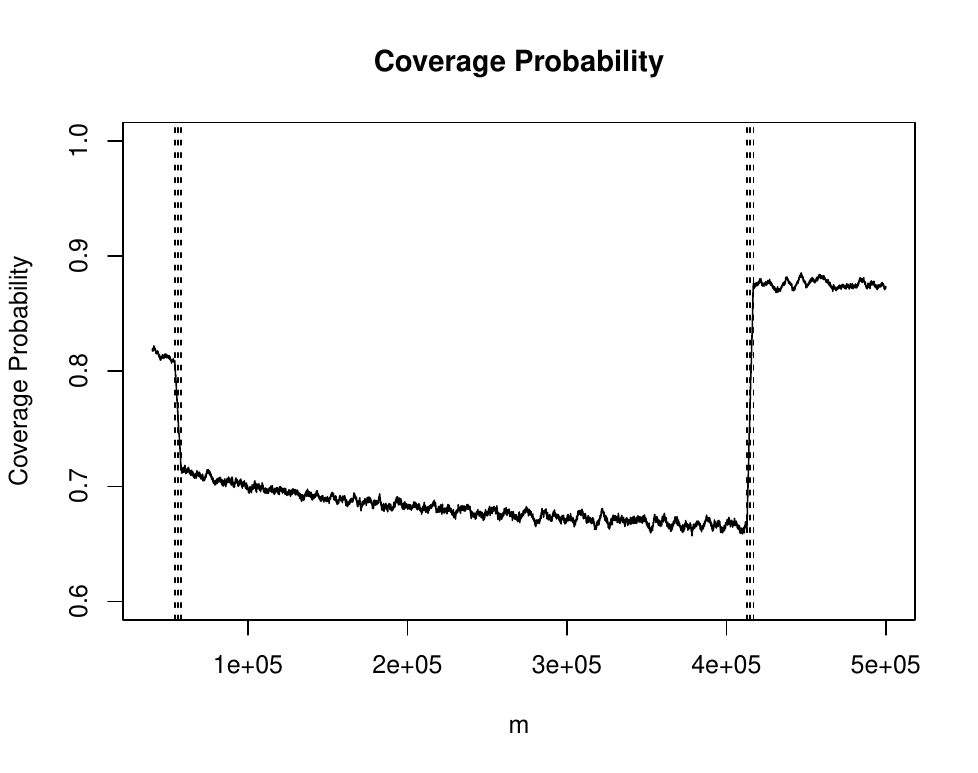}
    \caption{Plot of $p_{m, j, 1, \alpha^\star}$ with $j = 2000$ against $m$ for $m \in [4\cdot 10^4, 5\cdot 10^5]$. The first three vertical lines indicate $U^{-1}_{1,\alpha^\star}(1) - j$, $U^{-1}_{1,\alpha^\star}(1)$, and
    $U^{-1}_{1,\alpha^\star}(1) + j$, and the last three vertical lines indicate $L^{-1}_{1,\alpha^\star}(4) - j$, $L^{-1}_{1,\alpha^\star}(4)$, and
    $L^{-1}_{1,\alpha^\star}(4) + j$.}
    \label{fig:CR}
\end{figure}

The above results indicate that an atypical sequences of integers in the context of interval estimate occurs at around the inclusion and exclusion points. To visualize the results for different values of $z_{\alpha/2}$, here, we pay special attention to sequences of integers not exceeding the upper limit of \eqref{eqn:int0} for various values of $z_{\alpha/2}$. Specifically, let $s_{i,\alpha}$ denote the $i$-th occurrence of $1$ in $I(\omega(m) > U_{1,\alpha}(m))$, $m \geq 3$. That is,
\begin{align*}
     s_{i,\alpha} = \min\left\{n  \colon \sum_{m = 3}^n I(\omega(m) > U_{1,\alpha}(m)) = i\right\}.
\end{align*}
Now, let us consider $\log(s_{i+1,\alpha}/s_{i,\alpha})$, $i \in \mathbb{N}$, which measures the separation between the $i$-th and $(i+1)$-th occurrence of $m$ such that $\omega(m)$ exceeds $U_{1,\alpha}(m)$. Figure~\ref{fig:sep3D} illustrates how $\log(s_{i+1,\alpha}/s_{i,\alpha})$, $i \leq 10000$, change when $z_{\alpha/2} \in [0.5, 1.5]$. The 3D plot suggests that $\log(s_{i+1,\alpha}/s_{i,\alpha})$ forms a staircase, with the locations of sudden interventions caused by the inclusion points changing smoothly across different values of $z_{\alpha/2}$. In other words, the intervention phenomenon caused by the inclusion points is ubiquitous in different values of $z_{\alpha/2}$. At the same time, no other sequences of integers seem to cause a sudden shift in $\log(s_{i+1,\alpha}/s_{i,\alpha})$. To remove the interventions caused by the inclusion and exclusion points, it becomes necessary to modify the interpretation of the interval estimate and coverage probability, which we will discuss in Section \ref{sec:mod}.

\begin{figure}
\includegraphics[width=1\textwidth]{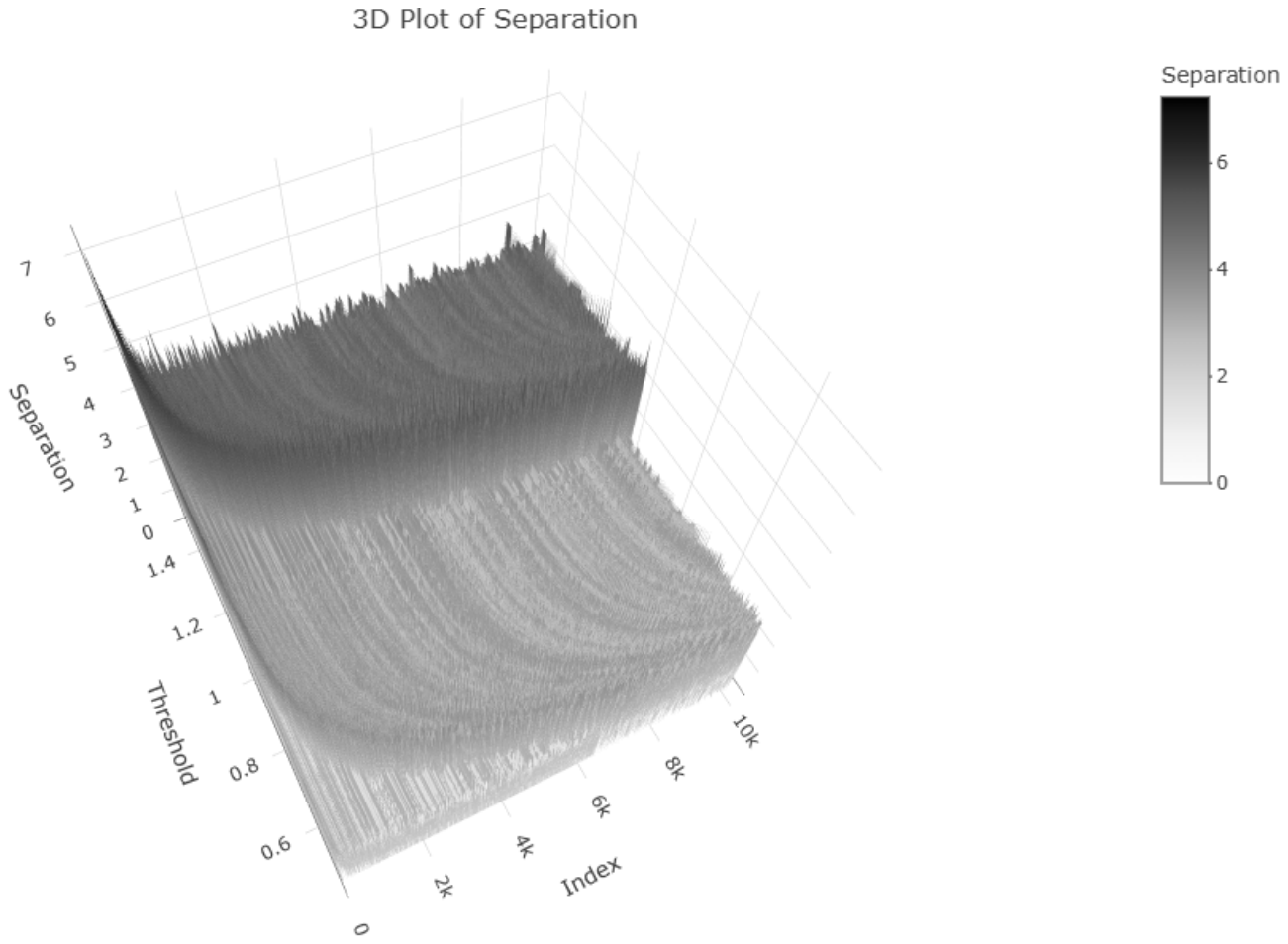}
    \caption{3D Plot of the $\log(s_{i+1,\alpha}/s_{i,\alpha})$ (Separation) against $i$ (Index) with $z_{\alpha/2} \in [0.5, 1.5]$ (Threshold).}
    \label{fig:sep3D}
\end{figure}

\section{Modifying the interpretation of the interval estimate and coverage probability}\label{sec:mod}

In Section \ref{sec:prob}, we observed that the coverage probability changes dramatically after the inclusion and exclusion points of the interval estimate. In other words, integers around these points may be considered atypical when the conventional definition of coverage probability is applied. However, because the sudden jumps in coverage probabilities around these points make the interval estimate unreliable, we consider modifying the interpretation of the interval estimate first by again examining Billingsley's interval estimate for $\omega$ in the vicinity of $10^{70}$, which is given by
\begin{align}\label{eqn:bill}
    [L_{1,\alpha^\star}(10^{70}), U_{1,\alpha^\star}(10^{70})] \approx [3, 7].
\end{align} 

Table~\ref{tab:ie} shows that two such points, namely, $L^{-1}_{1,\alpha^\star}(3)$ and $U^{-1}_{1,\alpha^\star}(7)$, occur relatively close to $10^{70}$, and that likely led to Billingsley's choice of $m = 10^{70}$. However, as $L^{-1}_{1,\alpha^\star}(3) < 10^{70}$, in the context of the coverage probability calculation, $[L_{1,\alpha^\star}(10^{70}), U_{1,\alpha^\star}(10^{70})] \approx [3.0535, 7.1115]$ is the same as that of $[4,7]$, which likely results in having the coverage probability lower than the nominal coverage probability of $1 - \alpha^\star \approx 0.6319$. To alleviate the problem, a more natural approach is to take into consideration the fractional part of the lower and upper bound of the interval estimate $[3.0535, 7.1115]$ using fuzzy set theory.

Here, suppose that we have an interval estimate $[L_{\lambda,\alpha}(m), U_{\lambda,\alpha}(m)]$ with the ceiling of the lower bound $\lceil L_{\lambda,\alpha}(m) \rceil$ and the floor of the upper bound $\lfloor U_{\lambda,\alpha}(m)\rfloor$. Also, let $p^L_{\lambda, \alpha}(m) = \lceil L_{\lambda,\alpha}(m) \rceil - L_{\lambda,\alpha}(m)$ and $p^U_{\lambda, \alpha}(m) = U_{\lambda,\alpha}(m) - \lfloor U_{\lambda,\alpha}(m) \rfloor$. Then, we interpret the interval estimate as follows.
\begin{itemize}
\item We assume that, $100p^L_{\lambda, \alpha}(m)\%$ of time, $\lfloor L_{\lambda,\alpha}(m) \rfloor$ is included in the interval estimate. Similarly, we assume that, $100p^U_{\lambda, \alpha}(m)\%$ of time, $\lceil U_{\lambda,\alpha}(m) \rceil$ is included in the interval estimate.
\end{itemize}
Also, for convenience, when $\lceil L_{\lambda,\alpha}(m) \rceil = \lfloor U_{\lambda,\alpha}(m)\rfloor + 1$, we may further assume that at least one of $\lceil L_{\lambda,\alpha}(m) \rceil$ and $\lfloor U_{\lambda,\alpha}(m)\rfloor$ is included in the interval estimate in practice. The above interpretation is more natural in the context of Billingsley's interval estimate, as the lower bound $L_{1,\alpha^\star}(10^{70})$ is only slightly greater than $3$ so that the interval estimate almost includes $3$, making the approximation in \eqref{eqn:bill} valid.

The interval estimate with the modified interpretation above, which we call the \emph{fuzzy} interval estimate (cf., \cite{geyer2005fuzzy}), also affects the coverage probability calculation. The new coverage probability calculation corresponding to \eqref{eqn:cr}, which we call the fuzzy coverage probability, is given by
\begin{align*}
    \tilde{p}_{m, j, \lambda, \alpha} = \frac{1}{2j+1}\sum_{t=m-j}^{m+j}\left[I^M_{\lambda, \alpha}(t) + p^L_{\lambda, \alpha}(t)I^L_{\lambda, \alpha}(t) + p^U_{\lambda, \alpha} I^U_{\lambda, \alpha}(t)\right],
\end{align*}
where 
\begin{align*}
I^M_{\lambda, \alpha}(t) &= I(\lceil L_{\lambda,\alpha}(t) \rceil \leq \lfloor U_{\lambda,\alpha}(t) \rfloor)I(\omega(t) \in [\lceil L_{\lambda,\alpha}(t) \rceil, \lfloor U_{\lambda,\alpha}(t) \rfloor]),\\ I^L_{\lambda, \alpha}(t) &= I(\omega(t) = \lfloor L_{\lambda,\alpha}(t) \rfloor),\\ 
I^U_{\lambda, \alpha}(t) &= I(\omega(t) = \lceil U_{\lambda,\alpha}(t) \rceil).
\end{align*}
Similarly, the fuzzy coverage probability calculation corresponding to \eqref{eqn:asympcr} is given by
\begin{align}\label{eqn:asympcrmod}
\tilde{p}^\star_{m,\lambda,\alpha} = \frac{1}{m}\sum_{d=1}^\infty \hat{\rho}_d(m)\left[I^M_{d, \lambda, \alpha}(m) + p^L_{\lambda, \alpha}(m)I^L_{d, \lambda, \alpha}(m) + p^U_{\lambda, \alpha} I^U_{d, \lambda, \alpha}(m) \right],
\end{align}
where
\begin{align*}
I^M_{d, \lambda, \alpha}(m) &= I(\lceil L_{\lambda,\alpha}(m) \rceil \leq \lfloor U_{\lambda,\alpha}(m) \rfloor)I(d \in [\lceil L_{\lambda,\alpha}(m) \rceil, \lfloor U_{\lambda,\alpha}(m) \rfloor]),\\ I^L_{d, \lambda, \alpha}(m) &= I(d = \lfloor L_{\lambda,\alpha}(m) \rfloor),\\ 
I^U_{d, \lambda, \alpha}(m) &= I(d = \lceil U_{\lambda,\alpha}(m) \rceil).
\end{align*}

To identify potentially promising interval estimates for $\omega$, let us compare the fuzzy coverage probabilities of five different interval estimates for $\omega$. The first three are $[L_{\lambda, \alpha}(m), U_{\lambda, \alpha}(m)]$ with $\lambda = 1/2$, $3/4$, and $1$ as each one of them has a favorable property. The fourth interval estimate, which is based on the asymptotic coverage probability result \eqref{eqn:asympcrmod}, is given in the form 
\begin{align}\label{eqn:ciasymp}
    [\ell_2(m) + 1 -\kappa_{\alpha}(m), \ell_2(m) + 1 +\kappa_{\alpha}(m)],
\end{align}
where $\kappa_{\alpha}(m)$ is set so that
\begin{align*}
    \tilde{p}^\star_{m,\lambda,\alpha} = 1 - \alpha
\end{align*}
when $L_{\lambda,\alpha}(m) = \ell_2(m) + 1 -\kappa_{\alpha}(m)$ and $U_{\lambda,\alpha}(m) = \ell_2(m) + 1 + \kappa_{\alpha}(m)$. We call the interval estimate \eqref{eqn:ciasymp} the Poisson interval estimate for $\omega$. Note that the Poisson interval estimate is centered at $\ell_2(m) + 1$, representing the mean of the shifted Poisson distribution with the probability mass function $\hat{\rho}_d(m)/m$. Finally, the fifth interval estimate can be derived from Corollary \ref{cor:m2} with $f_{\mu, \lambda} = \ell_2$, $f_{\sigma, \lambda} = \omega$, and $\lambda = 1$. To be more specific, because
\begin{align*}
\lim_{n \to \infty}\mathbb{P}_n\left(\left\{\left(\frac{\omega - \ell_2}{\sqrt{\omega}}\right)^2 \leq z_{\alpha/2}^2 \right\}\;\cap\; \left\{J \in \mathbb{N}_{3,n}\right\}\right) = 1-\alpha,
\end{align*}
by following \cite{politis2024studentization}, the lower and upper bound of the interval estimate for $\omega$ in the vicinity of $m$ are given by solving the quadratic equation
\begin{align*}
    \left(\frac{\omega(m) - \ell_2(m)}{\sqrt{\omega(m)}}\right)^2 = z_{\alpha/2}^2
\end{align*}
for $\omega(m)$. That is, for a sufficiently large $m$, the interval estimate for $\omega$ in the vicinity of $m$ is given by
\begin{align}\label{eqn:ciwilson}
    \left[\ell_2(m) + \frac{z_{\alpha/2}^2}{2} - z_{\alpha/2}\sqrt{\ell_2(m) + \frac{z_{\alpha/2}^2}{4}},
    \ell_2(m) + \frac{z_{\alpha/2}^2}{2} + z_{\alpha/2}\sqrt{\ell_2(m) + \frac{z_{\alpha/2}^2}{4}}\right].
\end{align}
As a similar interval estimate called the score interval for the binomial proportion is originally suggested by \cite{wilson1927probable}, we call \eqref{eqn:ciwilson} the score interval estimate of $\omega$. It is noted in \cite{politis2024studentization} that the score interval estimate may perform better than the other types of interval estimates based on standardization or variance stabilization (i.e., those with $\lambda = 1$ or $\lambda = 1/2$) for common distributions such as the Poisson distribution.

Figure \ref{fig:allcr} displays the fuzzy coverage probabilities for all the five interval estimates of $\omega$ for $m$ in the vicinity of $10^a$, $a = 5, 6, \ldots, 14$, at $\alpha = \alpha^\star$, corresponding to the nominal coverage probability of approximately $0.6319$. As expected, the fuzzy coverage probability eliminates sudden jumps in the coverage probabilities caused by the inclusion and exclusion points. Also, the fuzzy coverage probabilities of all the five interval estimates clearly exceed the nominal coverage probability. In other words, these interval estimates are called \emph{conservative}. The most conservative interval estimate is the score interval estimate, followed by the three interval estimates of the form $[L_{\lambda, \alpha}(m), U_{\lambda, \alpha}(m)]$ in \eqref{eqn:int11} with $\lambda \in \{1/2, 3/4, 1\}$. In fact, these three interval estimates have nearly identical performances. On the other hand, the Poisson interval estimate performs much better, with the fuzzy coverage probability of approximately $0.73$ is well within the limits given by Bradley's liberal criterion. 

\begin{figure}[htbp]
    \centering
\includegraphics[width=0.85\textwidth]{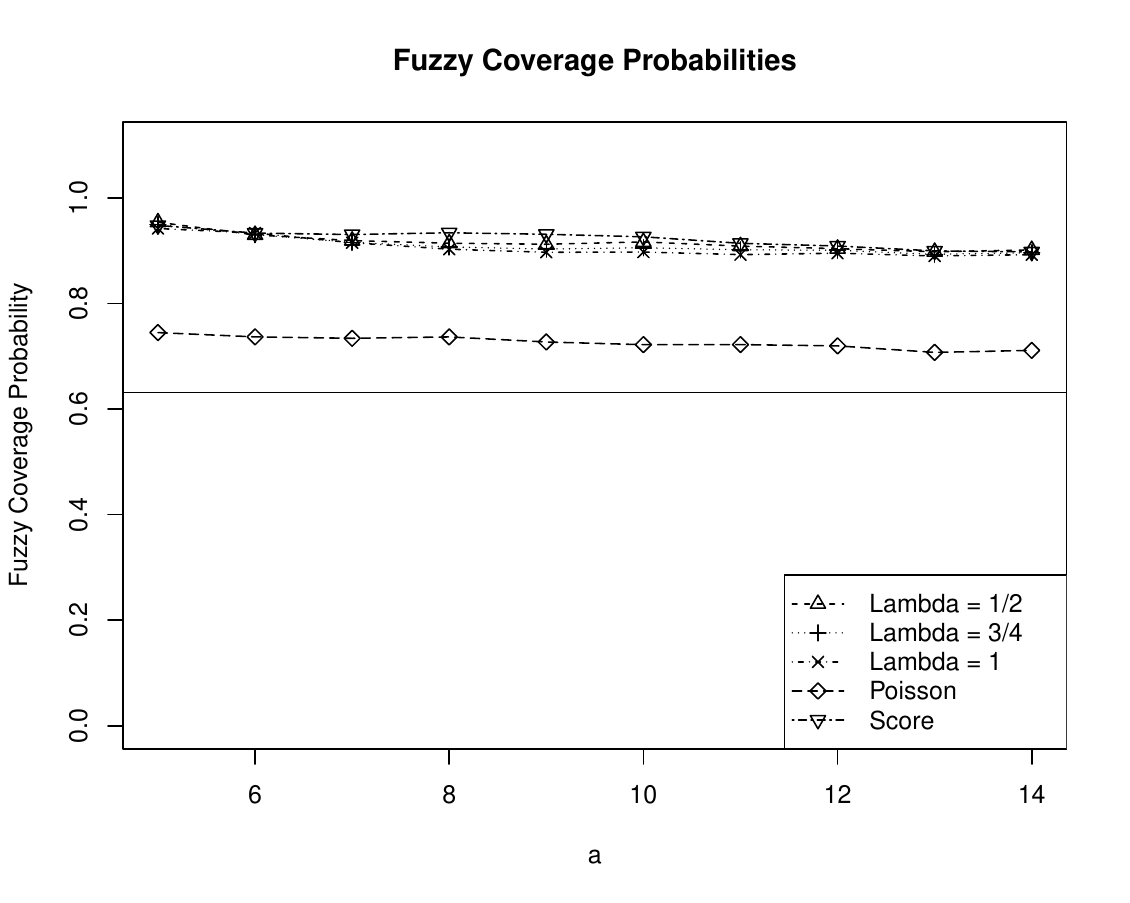}
    \caption{Plot of fuzzy coverage probabilities for the five interval estimates of $\omega$ in the vicinity of $m = 10^a$, $a = 5, 6, \ldots, 14$, at $\alpha = \alpha^\star$. The horizontal line indicates the nominal coverage probability.}
    \label{fig:allcr}
\end{figure}

Interestingly, the first three interval estimates of $\omega$ derived from \eqref{eqn:int11} with $\lambda \in \{1/2, 3/4, 1\}$ can be improved by adjusting the center. For example, by replacing the first factor $\ell_2(m)$ with $\ell_2(m) - 1$ in both the lower and upper bounds of the interval estimates (and not replacing $\sqrt{\ell_2(m)}$ with anything else), the coverage probabilities become much closer $1-\alpha^\star$ (see Figure \ref{fig:adjcr}). However, as the value of $-1$ is chosen somewhat arbitrarily, a more systematic statistical approach to adjust the center (i.e., mean) and scale (i.e., standard deviation) of $\omega$ may be more desirable. To accomplish our aim, in Section \ref{sec:train}, we further modify the interval estimates assuming that some information about $\omega$ for smaller $m$ is available for estimating these quantities. 

\begin{figure}[htbp]
    \centering
\includegraphics[width=0.85\textwidth]{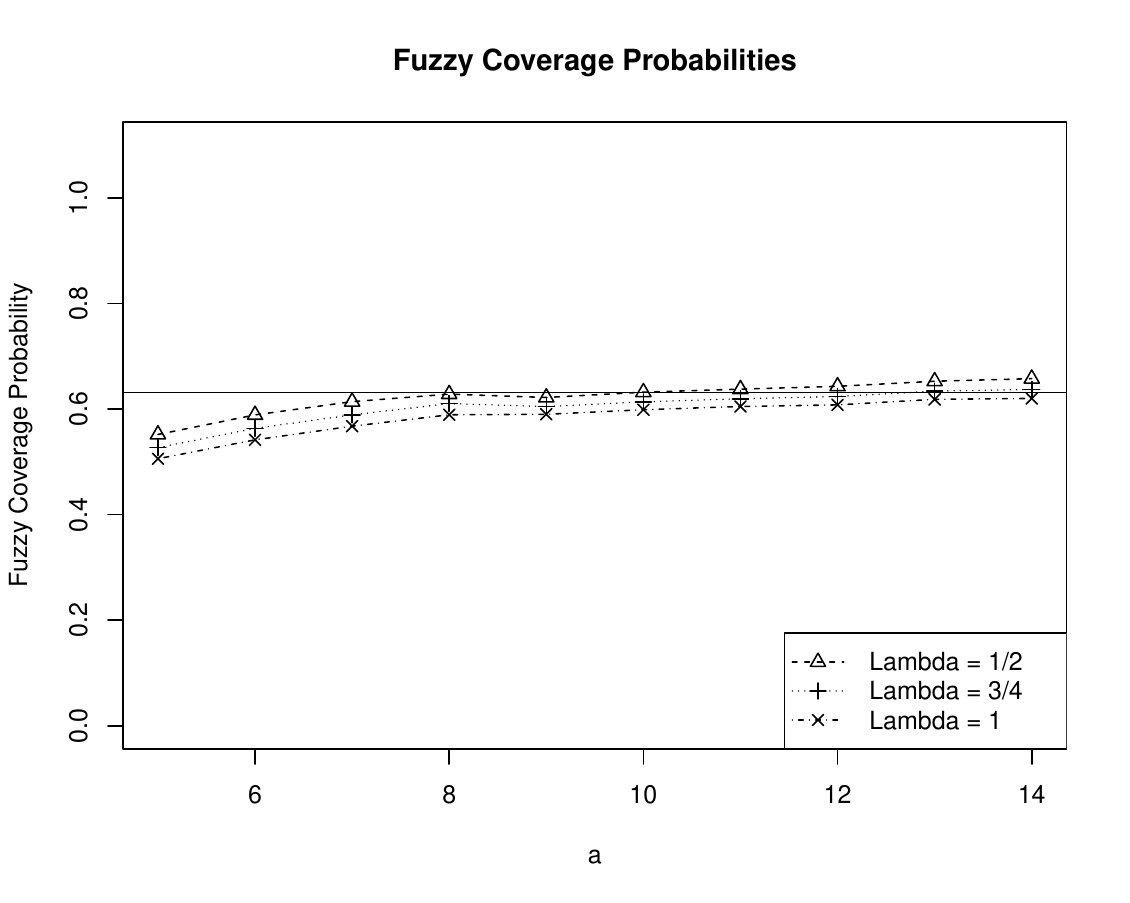}
    \caption{Plot of fuzzy coverage probabilities for the three adjusted interval estimates of $\omega$ in the vicinity of $m = 10^a$, $a = 5, 6, \ldots, 14$, at $\alpha = \alpha^\star$. The horizontal line indicates the nominal coverage probability.}
    \label{fig:adjcr}
\end{figure}

As a remark, it is also of interest to investigate cases where both $L_{\lambda,\alpha}(m)$ and $U_{\lambda,\alpha}(m)$ are (almost) integers. That way, the resulting interval estimate can be interpreted in a straightforward manner without relying on the fuzzy set theory. For example, take $z_{\alpha/2} = 2$ and $\ell_2(m) \approx u^2$ for some positive integer $u$. Then, when $\lambda = 1$, 
\begin{align*}
    [L_{1, \alpha}(m), U_{1, \alpha}(m)] \approx [u(u-2), u(u+2)],
\end{align*}
and when $\lambda = 1/2$,
\begin{align*}
    [L_{1/2, \alpha}(m), U_{1/2, \alpha}(m)] \approx [(u-1)^2, (u+1)^2],
\end{align*}
both of which have width $4u$ and their lower and upper bounds differ by just $1$ each. Moreover, the score interval estimate gives
\begin{align*}
    [L_{1/2, \alpha}(m), U_{1/2, \alpha}(m)] &\approx [u^2 - 2\sqrt{u^2+1} + 2, u^2 + 2\sqrt{u^2+1} + 2]\\
    &\approx [(u-1)^2 + 1, (u + 1)^2 + 1]
\end{align*}
for a large $u$, so that its lower and upper bound differ by just $1$ each (approximately) from the interval estimate with $\lambda = 1/2$.

Take, for example, $m = 10^{10^{3421914}}$ so that $\ell_2(m) \approx 2807^2 = 7879249$. Then, around $95\%$ ($95.45\%$ to be more precise) of the integers in the vicinity of $m = 10^{10^{3421914}}$ have $7873635$ to $7884863$ distinct prime divisors based on the interval estimate with $\lambda = 1$, $7873636$ to $7884864$ distinct prime divisors based on the interval estimate with $\lambda = 1/2$, and $7873637$ to $7884865$ distinct prime divisors based on the score interval estimate. In contrast, the Poisson interval estimate gives $7873637$ to $7884863$ distinct prime divisors, which are slightly narrower in width than the other three estimates. Thus, even for large numbers such as $10^{10^{3421914}}$, the Poisson interval estimate is considered superior. Lastly, these interval estimates also highlight the rareness of $10^{10^{3421914}}$ itself, which only has $2$ distinct prime divisors.

\section{Training the interval estimates}
\label{sec:train}

As the five interval estimates of $\omega$ examined in Section \ref{sec:mod} are conservative when $m$ is small (i.e., $m \in [10^5, 10^{14}]$), we investigate the possibility of statistically improving their performances by assuming that some information about $\omega$ for smaller $m$ is available. In particular, we show that the reliability of these interval estimates can be improved in terms of fuzzy coverage probability for $m \in [10^5, 10^{14}]$ by utilizing the estimated mean and standard deviation based on smoothed $\omega(m)$ for $m \in [10^4, 10^6]$ with an increment of $10^3$ between $10^4$ and $10^6$. The increment of $10^3$ was chosen to reduce the computational burden without losing the overall accuracy. For the choice of $\lambda$, we use $\lambda = 1/2$, $3/4$, and $1$ as before as these three choices provide some favorable properties. In addition, we modify the Poisson and score interval estimate using a similar approach.

Our systematic statistical approach relies on estimating the local adjustment functions $f_{\mu, \lambda}$ and $f_{\sigma, \lambda}$ (presented in Corollary \ref{cor:m2o}) in the vicinity of some given $m$. To do so, the following steps are taken, assuming $\lambda > 0$.
\begin{enumerate}
    \item Smooth $\omega^\lambda$ using a moving average filter around $m$, and call it $\tilde{\omega}_\lambda$. 
    \item Find a power $q_{\mu, \lambda}$ so that the linear correlation between $\tilde{\omega}_\lambda^{q_{\mu, \lambda}/\lambda}$ and $\ell_2$ is maximized.
    \item To obtain a local estimate of $f_{\mu, \lambda}$, find the (non-linear) least squares estimate of $(\beta_{0,\lambda}, \beta_{1,\lambda})$, denoted by $(\hat{\beta}_{0,\lambda}, \hat{\beta}_{1,\lambda})$, which minimizes the sum of squared errors between $\tilde{\omega}_\lambda$ and $(\beta_{0,\lambda} + \beta_{1,\lambda}\ell_2)^{\lambda/q_{\mu, \lambda}}$, for some range of $m$.
    \item Let $\hat{f}_{\mu, \lambda} = (\hat{\beta}_{0,\lambda} + \hat{\beta}_{1,\lambda}\ell_2)^{1/q_{\mu, \lambda}}$ be the local estimate of $f_{\mu, \lambda}$ around $m$. Then, estimate the standard deviation of $(\omega^\lambda - \hat{f}_{\mu, \lambda}^\lambda)/\lambda$ and call it $\tilde{\sigma}_\lambda$. 
    \item Find a power $q_{\sigma, \lambda}$ so that the linear correlation between $\tilde{\sigma}_\lambda^{q_{\sigma, \lambda}}$ and $\ell_2$ is maximized. 
    \item To obtain a local estimate of $f_{\sigma, \lambda}$, find the (non-linear) least squares estimate of $(\gamma_{0,\lambda}, \gamma_{1,\lambda})$, denoted by $(\hat{\gamma}_{0,\lambda}, \hat{\gamma}_{1,\lambda})$, which minimizes the sum of squared errors between $\tilde{\sigma}_\lambda$ and $(\gamma_{0,\lambda} + \gamma_{1,\lambda}\ell_2)^{1/q_{\sigma, \lambda}}$ for the same range of $m$ as before.
    \item Let $\hat{f}_{\sigma, \lambda} = (\hat{\gamma}_{0,\lambda} + \hat{\gamma}_{1,\lambda}\ell_2)^{1/q_{\sigma, \lambda}}$ be the local estimate of $f_{\sigma, \lambda}$ in the vicinity of $m$. Assuming that 
    \begin{align*}
        \frac{\omega^\lambda - \hat{f}_{\mu, \lambda}^\lambda}{\lambda \hat{f}_{\sigma, \lambda}}
    \end{align*}
    is approximately standard normal in the vicinity of $m$, the lower and upper bound for the $100(1-\alpha)\%$ interval estimate of $\omega$ in the vicinity of a sufficiently large $m$ is given by $[\hat{L}_{\lambda, \alpha}(m), \hat{U}_{\lambda, \alpha}(m)]$, where
    \begin{align*}
    \hat{L}_{\lambda, \alpha}(m) = ([\hat{f}_{\mu,\lambda}(m)]^{\lambda} - z_{\alpha/2}\lambda\hat{f}_{\sigma, \lambda}(m))^{1/\lambda}
    \end{align*}
    and
    \begin{align*}
    \hat{U}_{\lambda, \alpha}(m) = ([\hat{f}_{\mu,\lambda}(m)]^{\lambda} + z_{\alpha/2}\lambda\hat{f}_{\sigma, \lambda}(m))^{1/\lambda}.
    \end{align*}
\end{enumerate}

To illustrate the effectiveness of the algorithm above, in the first step, $\omega^\lambda$ is smoothed using the moving average filter with $j = 2000$ points before and after $m$, $m \in [10^4, 10^6]$ with an increment of $10^3$. That is,
\begin{align*}
\tilde{\omega}_\lambda(m) = \frac{1}{2j+1}\sum_{t=m-j}^{m+j} [\omega(t)]^\lambda
\end{align*}
is the smoothed $\omega^\lambda$ at $m$. Smoothing is an effective approach to estimate the underlying pattern in an unknown function especially when the function has a low signal-to-noise ratio. Now, Figure \ref{fig:smooth_means} shows that $\tilde{\omega}_\lambda^{1/\lambda}$ is almost perfectly linearly correlated with $\ell_2$ for $\lambda = 1/2$, $3/4$, and $1$ with the sample correlation coefficient $> 0.999$, suggesting that $\tilde{\omega}_\lambda$ is well approximated by $(\beta_{0, \lambda} + \beta_{1,\lambda}\ell_2)^\lambda$. That also implies that $q_{\mu, \lambda} = 1$ works well in this range. Using the least squares estimates $(\beta_{0,\lambda}^{(0)}, \beta_{1,\lambda}^{(0)})$ when regressing $\ell_2$ on $\tilde{\omega}_\lambda^{1/\lambda}$ as the initial estimates, the nonlinear least squares estimates of $(\beta_{0,\lambda}, \beta_{1,\lambda})$, denoted by $(\hat{\beta}_{0,\lambda}, \hat{\beta}_{1,\lambda})$, which minimize the sum of squared errors between $\tilde{\omega}_\lambda$ and $(\beta_{0, \lambda} + \beta_{1,\lambda}\ell_2)^\lambda$, are given in Table \ref{tab:lse}.

\begin{figure}[htbp]
    \centering
\includegraphics[width=1\textwidth]{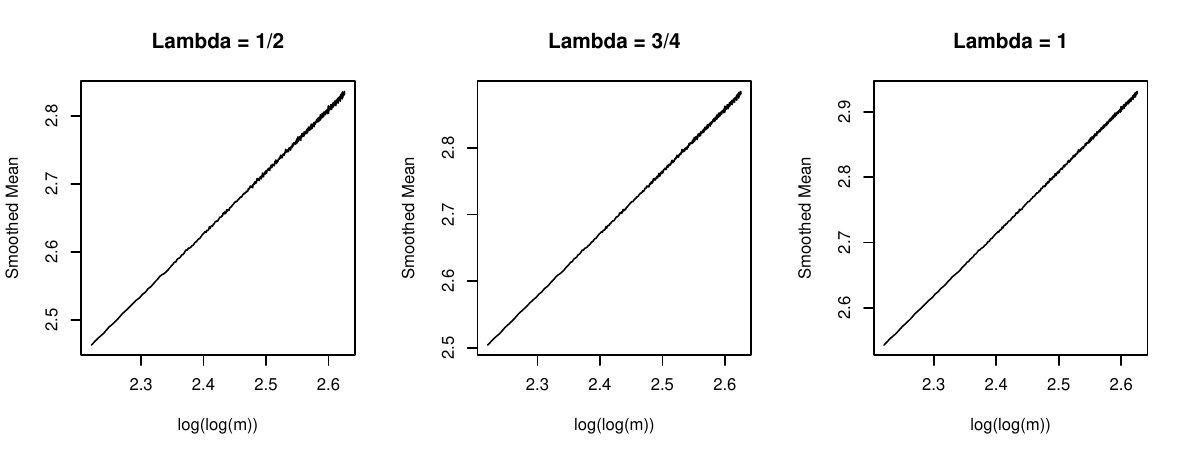}
    \caption{Plots of smoothed means $[\tilde{\omega}_\lambda(m)]^{1/\lambda}$ against $\ell_2(m)$, $m \in [10^4, 10^6]$, for $\lambda \in \{1/2, 3/4, 1\}$.}
    \label{fig:smooth_means}
\end{figure}

\begin{table}[htbp]
\centering
\begin{tabular}{|c|c|c|c|c|} 
 \hline
 $\lambda$ & $\hat{\beta}_{0,\lambda}$ & $\hat{\beta}_{1,\lambda}$ & $\hat{\gamma}_{0,\lambda}$ & $\hat{\gamma}_{1,\lambda}$ \\ 
 \hline
  $1/2$ & $0.4300$ & $0.9152$ & $-0.0109$ & $0.1498$ \\ 
 \hline
 $3/4$ & $0.4284$ & $0.9343$ & $0.0196$ & $0.2695$\\  
 \hline
 $1$ & $0.4270$ & $0.9527$ & $0.0499$ & $0.3677$\\   
 \hline 
\end{tabular}
\caption{$\hat{\beta}_{0,\lambda}$, $\hat{\beta}_{1,\lambda}$, $\hat{\gamma}_{0,\lambda}$, and $\hat{\gamma}_{1,\lambda}$ for $\lambda \in \{1/2, 3/4, 1\}$.}
\label{tab:lse}
\end{table}

Let $\hat{f}_{\mu, \lambda} = \hat{\beta}_{0,\lambda} + \hat{\beta}_{1,\lambda}\ell_2$. In the next step, the standard deviation of $(\omega^\lambda - \hat{f}_{\mu, \lambda}^\lambda)/\lambda$ is estimated by
\begin{align*}
\tilde{\sigma}_\lambda(m) = \left[\frac{1}{2j+1}\sum_{t=m-j}^{m+j} \left(\frac{[\omega(t)]^\lambda - [\hat{f}_{\mu, \lambda}(t)]^\lambda}{\lambda}\right)^2\right]^{1/2}.
\end{align*}

Figure \ref{fig:smooth_sds} shows that $\tilde{\sigma}_\lambda^{1/\lambda}$ is almost perfectly linearly correlated with $\ell_2$ as well (with the sample correlation coefficient $> 0.943$), suggesting that $\tilde{\sigma}_\lambda$ is well approximated by $(\gamma_{0, \lambda} + \gamma_{1,\lambda}\ell_2)^\lambda$ for some $(\gamma_{0, \lambda}, \gamma_{1, \lambda})$. In other words, $q_{\sigma, \lambda} = 1/\lambda$ works well in this range. Therefore, by taking the same step as before, the nonlinear least squares estimates of $(\gamma_{0,\lambda}, \gamma_{1,\lambda})$, denoted by $(\hat{\gamma}_{0,\lambda}, \hat{\gamma}_{1,\lambda})$, are also given in Table \ref{tab:lse}. Lastly, by letting $\hat{f}_{\sigma, \lambda} = (\hat{\gamma}_{0,\lambda} + \hat{\gamma}_{1,\lambda}\ell_2)^\lambda$, we obtain the interval estimate $[\hat{L}_{\lambda, \alpha}(m), \hat{U}_{\lambda, \alpha}(m)]$ for $\lambda = 1/2$, $3/4$, and $1$.

\begin{figure}[htbp]
    \centering
\includegraphics[width=1\textwidth]{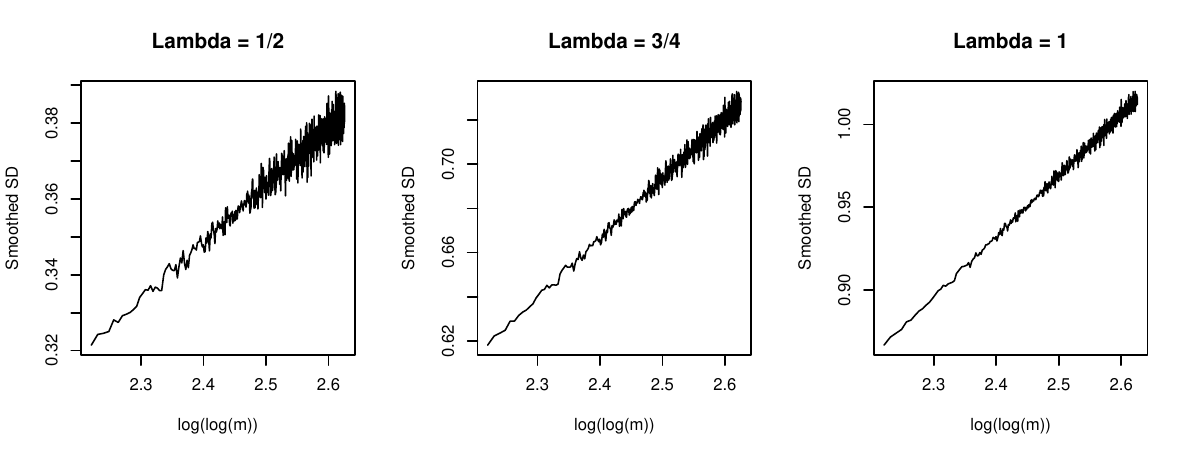}
    \caption{Plots of smoothed standard deviations (SDs) $\tilde{\sigma}_\lambda^{1/\lambda}$ against $\ell_2$, $m \in [10^4, 10^6]$, for $\lambda \in \{1/2, 3/4, 1\}$.}
    \label{fig:smooth_sds}
\end{figure}

In addition, to create a trained version of the score interval estimate, using the fact that $\tilde{\sigma}_1(m)$ is almost perfectly linearly correlated with $\tilde{\omega}_1(m)$ (with the sample correlation coefficient $> 0.990$), $\tilde{\sigma}^2_1(m)$ can also be well approximated by $\eta_0 + \eta_1\tilde{\omega}_1(m)$ for some $(\eta_0, \eta_1)$. Thus, the least squares estimates of $(\eta_0, \eta_1)$, denoted by $(\hat{\eta}_0, \hat{\eta}_1)$, are given by $(-0.1136, 0.3855)$. Then, assuming that
\begin{align*}
    \frac{\omega - \hat{f}_{\mu, \lambda}}{\hat{\eta}_0 + \hat{\eta}_1\omega}
\end{align*}
is approximately standard normal in the vicinity of $m$, we solve the quadratic equation 
\begin{align*}
    \left(\frac{\omega(m) - \hat{f}_{\mu, \lambda}(m)}{\hat{\eta}_0 + \hat{\eta}_1\omega(m)}\right)^2 = z_{\alpha/2}^2
\end{align*}
for $\omega(m)$. That is, the trained version of the score interval estimate for $\omega$ is given by $[\tilde{L}_{\lambda, \alpha}(m), \tilde{U}_{\lambda, \alpha}(m)]$, where
\begin{align*}
\tilde{L}_{\lambda, \alpha}(m) = \hat{f}_{\mu, \lambda}(m) + \frac{\hat{\eta}_1 z^2_{\alpha/2}}{2} - 
z_{\alpha/2}\sqrt{\hat{\eta}_0 + \hat{\eta}_1\hat{f}_{\mu, \lambda}(m) + \frac{\hat{\eta}^2_1 z^2_{\alpha/2}}{4}}
\end{align*}
and 
\begin{align*}
\tilde{U}_{\lambda, \alpha}(m) = \hat{f}_{\mu, \lambda}(m) + \frac{\hat{\eta}_1 z^2_{\alpha/2}}{2} + 
z_{\alpha/2}\sqrt{\hat{\eta}_0 + \hat{\eta}_1\hat{f}_{\mu, \lambda}(m) + \frac{\hat{\eta}^2_1 z^2_{\alpha/2}}{4}},
\end{align*}
assuming that $\hat{\eta}_0 + \hat{\eta}_1\hat{f}_{\mu, \lambda}(m) + \hat{\eta}^2_1 z^2_{\alpha/2}/4 \geq 0$. 

Furthermore, the Poisson interval estimate \eqref{eqn:ciasymp} may be trained as well. Let
\begin{align}\label{eqn:poispmftrain}
\hat{\rho}_{d, \beta_0, \beta_1}(m) = 
\frac{m\exp[-(\beta_0 + \beta_1\ell_2(m))][\beta_0 + \beta_1\ell_2(m)]^{d-1}}{(d-1)!}
\end{align}
and define $\tilde{p}^\star_{m,\lambda,\alpha, \beta_0, \beta_1}$ similarly as \eqref{eqn:asympcrmod} by replacing $\ell(m)$ with $\beta_0 + \beta_1\ell(m)$. Now, consider 
\begin{align*}
    [\hat{f}_{\mu, 1}(m) - \kappa_{\alpha, \hat{\beta}_0 - 1, \hat{\beta}_1}(m), \hat{f}_{\mu, 1}(m) + \kappa_{\alpha, \hat{\beta}_0 - 1, \hat{\beta}_1}(m)],
\end{align*}
where $\kappa_{\alpha, \hat{\beta}_0 - 1, \hat{\beta}_1}(m)$ is set so that that
\begin{align}\label{eqn:poisp}
\tilde{p}^\star_{m, \lambda, \alpha, \hat{\beta}_0 - 1, \hat{\beta}_1} = 1 - \alpha.
\end{align}
To see why $\hat{\beta}_0 - 1$ is chosen in \eqref{eqn:poisp}, note that the mean of the shifted Poisson distribution with the probability mass function $\hat{\rho}_{d, \hat{\beta}_0 - 1, \hat{\beta}_1}(m)/m$, $m \in \mathbb{N}$, is given by $\hat{f}_{\mu, 1}(m)$.

To assess the reliability of the trained interval estimates of $\omega$, both the in-sample (i.e., $m \leq 10^6$) and out-of-sample (i.e., $m > 10^6$) performances of the five trained interval estimates are examined in terms of fuzzy coverage probability for $m = 10^a$, $a = 5, 6, \ldots, 14$, at $\alpha = \alpha^\star$. Figure \ref{fig:modcrmod} indicates that the most reliable interval estimate in terms of the in-sample performance is the trained score interval estimate, but is closely followed by the trained Poisson interval estimate. In terms of the out-of-sample performance, sadly, as $m$ increases, the fuzzy coverage probability also increases for both of them, although the Poisson interval estimate has an almost exact fuzzy coverage probability for $m \in [10^7, 10^8]$. 

On the other hand, the other three trained interval estimates (with $\lambda = 1/2$, $3/4$, and $1$) perform nearly identically. Unfortunately, they all stay consistently well above the nominal coverage probability, making them conservative even after the training. Therefore, for these interval estimates, a simple centering adjustment considered in Section \ref{sec:mod} may be more effective especially for relatively small $m$. 

\begin{figure}[htbp]
    \centering
\includegraphics[width=1\textwidth]{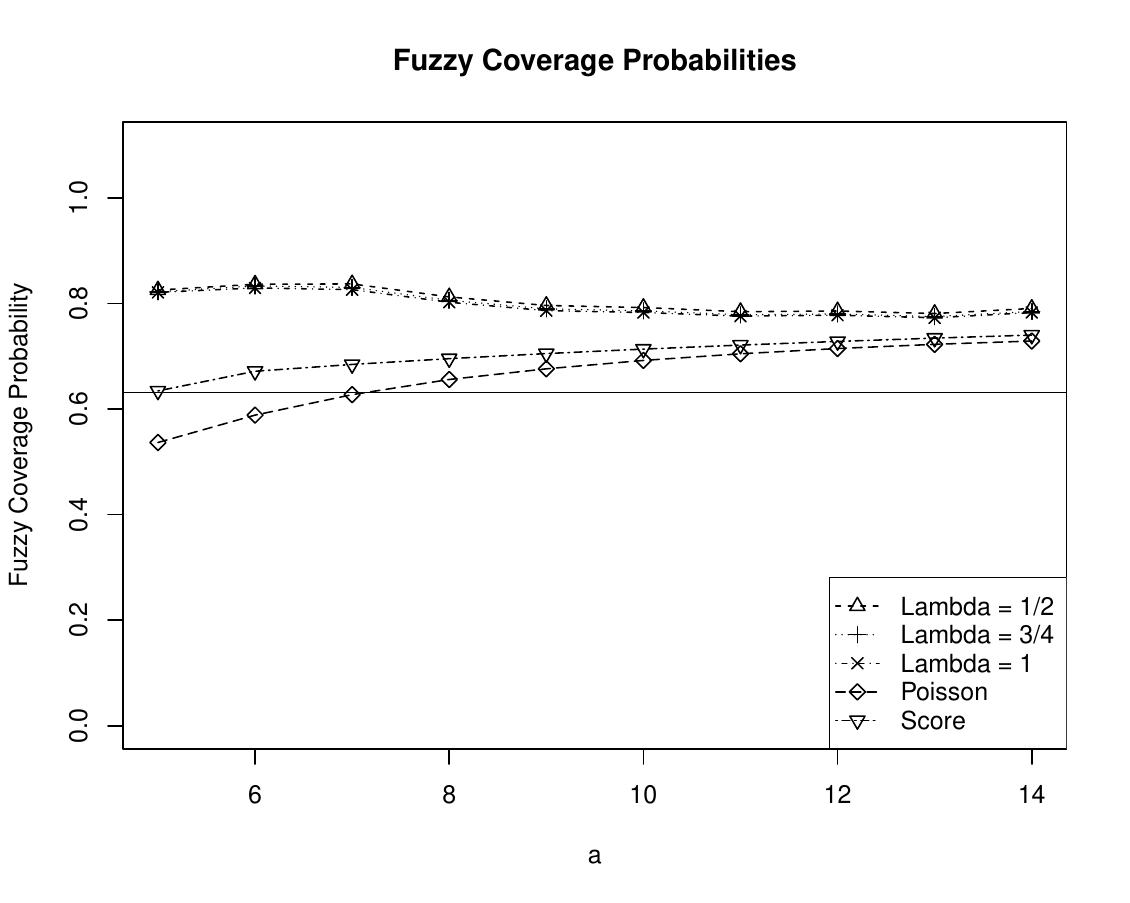}
    \caption{Plot of fuzzy coverage probabilities of the five trained interval estimates of $\omega$ in the vicinity of $m = 10^a$, $a = 5, 6, \ldots, 14$, at $\alpha = \alpha^\star$. The horizontal line indicates the nominal coverage probability.}
    \label{fig:modcrmod}
\end{figure}

\section{Theoretical results on the Erd\H{o}s-Pomerance theorem}\label{sec:ep}

Although we mainly examined transformations of the Erd\H{o}s-Kac theorem and their applications to interval estimation of $\omega$, the transformation technique is also applicable to other Erd\H{o}s-Kac-type theorems. Here, we discuss transformations of the Erd\H{o}s-Pomerance theorem \citep{erdos1985normal}. Let $\varphi$, defined as $\varphi(m) = m\prod_{p \vert m} (1 - p^{-1})$, $m \in \mathbb{N}$, denote Euler's totient function. Then, the Erd\H{o}s-Pomerance theorem shows the normality of the number of distinct prime divisors of Euler's totient function, $\omega(\varphi(m)) = (\omega \circ \varphi)(m)$, $m \in \mathbb{N}$, loosely speaking. Formally, it is given by
\begin{align}\label{eqn:thmep}
\frac{\omega \circ \varphi - [\ell_2(n)]^2/2}{[\ell_2(n)]^{3/2}/\sqrt{3}} \Rightarrow \Phi,
\end{align}
where we may set $\mathbb{M}_n = \mathbb{N}_{3,n}$ as before. Although \eqref{eqn:thmep}, directly taken from \cite{erdos1985normal}, contains $n$ in the expression, we may omit $n$ and conclude that
\begin{align}\label{eqn:thmep2}
\frac{\omega \circ \varphi - \ell_2^2/2}{\ell_2^{3/2}/\sqrt{3}} \Rightarrow \Phi
\end{align}
as well. To see why, by following \cite{billingsley1969central} for the similar derivation using the Erd\H{o}s-Kac theorem, it is equivalent to showing that 
\begin{align}\label{eqn:thmep3}
\lim_{n \to \infty}\frac{[\ell_2(n)]^2 - [\ell_2(m)]^2}{[\ell_2(n)]^{3/2}} = 0
\end{align}
if $\sqrt{n} < m \leq n$ in the case of the Erd\H{o}s-Pomerance theorem. By substituting $m$ with $\sqrt{n}$, a straightforward calculation shows that \eqref{eqn:thmep3} indeed holds, and hence \eqref{eqn:thmep2} follows as well.

Using \eqref{eqn:thmep2}, we now present results analogous to Theorems \ref{thm:mainomega}, \ref{thm:m1} and \ref{thm:m2general}. The proofs are omitted for these theorems as they are nearly identical to the ones presented in Sections \ref{sec:main} and \ref{sec:box}.
\begin{theorem}\label{thm:mainep}
Let $g$ be a (real-valued) function whose domain contains $[1-\delta, 1 + \delta]$ for some $\delta > 0$ with the assumptions that $g$ is differentiable at $1$ and that $g^\prime(1) \neq 0$. Then,
\begin{align*}
\frac{g(2\omega \circ \varphi/\ell_2^2) - g(1)}{2g^\prime(1)/\sqrt{3\ell_2}} \Rightarrow \Phi.
\end{align*}
\end{theorem}

\begin{theorem}\label{thm:m1ep} 
When $\lambda \neq 0$,
\begin{align*}
\frac{(2\omega \circ \varphi)^\lambda - \ell_2^{2\lambda}}{2\lambda \ell_2^{2\lambda - 1/2}/\sqrt{3}} \Rightarrow \Phi,
\end{align*}
and when $\lambda = 0$,
\begin{align*}
\frac{\log(2\omega \circ \varphi) - 2\log\ell_2}{2/\sqrt{3\ell_2}} \Rightarrow \Phi.
\end{align*}
\end{theorem}

\begin{theorem}\label{thm:m2ep} 
Let $f_{\mu} = \ell_2^2 + o(\ell_2^{3/2})\mbox{ pp}$ and $f_{\sigma} = \ell_2^{3/2} + o(\ell_2^{3/2})\mbox{ pp}$ be some local adjustment functions. Also, let $g$ be a (real-valued) function whose domain contains $[1-\delta, 1 + \delta]$ for some $\delta > 0$ with the assumptions that $g$ is differentiable at $1$ and that $g^\prime(1) \neq 0$. Then,
\begin{align*}
\frac{g(2\omega \circ \varphi/f_{\mu}) - g(1)}{2g^\prime(1)(f_\sigma/f_{\mu})/\sqrt{3}} \Rightarrow \Phi.
\end{align*}
\end{theorem}

Although Theorem \ref{thm:m1ep} is nearly identical to Theorem \ref{thm:m1}, Theorem \ref{thm:m1ep} indicates that $\lambda = 1/4$ is the variance-stabilizing power transformation for the Erd\H{o}s-Pomerance theorem. In particular,
\begin{align*}
2\sqrt{3}[(2\omega \circ \varphi)^{1/4} - \sqrt{\ell_2}] \Rightarrow \Phi.
\end{align*}
On the other hand, the optimal width of the interval estimate for $\omega \circ \varphi$ in the vicinity of $m$ is still achieved asymptotically at $\lambda = 3/4$ as $\ell(m) \to \infty$.

\section{Discussion and future work}
\label{sec:conc}

The Erd\H{o}s-Kac theorem, which reveals a surprising connection between the number of distinct prime divisors $\omega$ and the normal distribution, may be described from the probabilistic point of view. That naturally allows us to transform the theorem via standard methods in probability theory and mathematical statistics such as the delta method and Slutsky's theorem, and classical results in number theory such as Mertens' second theorem and the Hardy-Ramanujan theorem. Furthermore, when the Box-Cox transformation is considered, the transformed Erd\H{o}s-Kac theorem possesses favorable properties such as variance stabilization at $\lambda = 1/2$ and asymptotically optimal width of the resulting interval estimate for $\omega$ at $\lambda = 3/4$. In addition, a direct application of these theoretical results to the Erd\H{o}s-Pomerance theorem shows that variance stabilization is achieved at $\lambda = 1/4$ instead, although the asymptotically optimal width remains at $\lambda = 3/4$. Lastly, it is noted that all the theoretical results follow even if we replace $\omega$ with $\Omega$, which represents the total number of prime divisors with multiplicity counted \citep{billingsley1974probability,erdos1985normal}.

Despite its mathematical beauty, the interval estimates of $\omega$ derived from the Erd\H{o}s-Kac theorem and its transformations display significant jumps in conventional coverage probabilities at the inclusion and exclusion points. To overcome the issue, a fuzzy interval estimate interpretation is suggested, noting the discrete nature of $\omega$. The resulting fuzzy coverage probabilities show no significant jumps, although the fuzzy coverage probabilities themselves tends to be significantly higher than the nominal coverage probability for relatively small positive integers, revealing the conservative nature of such interval estimates.

To alleviate the conservativeness of these interval estimates, local adjustment functions are introduced in the numerator of denominator of the Box-Cox-transformed Erd\H{o}s-Kac theorem. Although their asymptotic properties are straightforward, it is a nontrivial task to estimate these functions. Nevertheless, our results indicate the possibility of statistically estimating these functions locally using linear transformations of $\ell_2$ or $\omega$ based on the available information about $\omega$ after a suitable power transformation. Although the trained interval estimates vary in their performances in terms of fuzzy coverage probability, the trained score and Poisson interval estimate seem particularly reliable as their fuzzy coverage probabilities stay only slightly above the nominal coverage probability in the out-of-sample case. In other words, these two interval estimates may provide relatively reliable out-of-sample interval estimates of $\omega$ in the short run. On the other hand, for the Box-Cox-transformed interval estimates, a simple centering adjustment may provide reasonable out-of-sample interval estimates of $\omega$ in the short run.

Although the main purpose of the article is on the transformation of the Erd\H{o}s-Kac theorem, the Poisson interval estimate of $\omega$ is found to be only somewhat conservative even without any training. In addition, another advantage of the Poisson interval estimate, which is that it tends to be narrower than the other interval estimates, is observed when $m \approx 10^{10^{3421914}}$ with $z_{\alpha/2} = 2$ is considered. Therefore, when constructing an interval estimate of $\omega$ in the vicinity of some $m$, it may be best to use the Poisson interval estimate with or without training, especially when $m$ is relatively small. To illustrate its advantage, for $m \approx 10^{70}$ with $\alpha = \alpha^\star$, the Poisson interval estimate gives $[4.5169, 7.6482]$, whose width is about $3.1313$. On the other hand, Billingsley's interval estimate of $[3.0535, 7.1115]$ has the width of about $4.0580$, which is noticeably larger than that of the Poisson interval estimate. Considering the fact that the interval estimate with $\lambda = 1$ tends to be highly conservative up to $m \approx 10^{14}$, it is likely that $[3.0535, 7.1115]$ is also too wide in terms of fuzzy coverage probability, making the interval estimate conservative. On the other hand, the Poisson interval estimate of $[4.5169, 7.6482]$ is probably more reliable.

Possible avenues for future work include applications of the transformation techniques to other results in probabilistic number theory, further generalizations of the transformed Erd\H{o}s-Kac theorem, efficient and accurate estimation of the local adjustment functions, and more research on the relationship between the Poisson distribution and $\omega$. By applying common techniques in probability theory and mathematical statistics such as the delta method and Slutsky's theorem to other Erd\H{o}s-Kac-type theorems, it may be possible to simplify their expressions through variance stabilization. Also, it is possible to generalize the transformed the Erd\H{o}s-Kac theorem and Erd\H{o}s-Pomerance theorem by applying techniques used in \cite{loyd2023dynamical}. Furthermore, by combining an appropriate transformation with efficient training of the local adjustment functions, the resulting interval estimates may become more useful for relatively small integers where the asymptotic results do not necessarily apply. Lastly, the fact that the Poisson interval estimate tends to outperform the other normal-based interval estimates derived from the Erd\H{o}s-Kac theorem and its modifications indicates that the Poisson approximation of $\omega$ seems to be a more natural choice than the normal approximation. Thus, a further investigation on the connection between the Poisson distribution and $\omega$ (e.g., \cite{harper2009two}) may lead to a deeper understanding of $\omega$ itself.

\section*{Acknowledgments}
The author would like to thank \'{A}rp\'{a}d B\'{e}nyi and Amites Sarkar for their helpful comments. 

\bibliographystyle{apalike} 
\bibliography{cas-refs}

\end{document}